\documentclass[11pt]{amsart}
\usepackage{graphicx, amsmath, amssymb, latexsym, cite, hyperref,color}
\newtheorem{theorem}{Theorem}[section]
\newtheorem{lemma}[theorem]{Lemma}
\newtheorem{proposition}[theorem]{Proposition}

\theoremstyle{definition}
\newtheorem{definition}[theorem]{Definition}

\newtheorem*{observation}{Observation}
\theoremstyle{remark}
\newtheorem{remark}[theorem]{Remark}

\numberwithin{equation}{section}

\newcommand{\R}{{\mathbb R}}

\newcommand{\Z}{{\mathbb Z}}

\newcommand{\N}{{\mathbb N}}
\newcommand{\T}{{\mathbb T_{j,\tau}}}

\newcommand{\ra}{{\textnormal{RapDec}}}

\newcommand{\supp}{{\operatorname{supp}}}

\begin{document}

\title[Microlocal Decoupling and Riemannian Distance Set]{Microlocal decoupling inequalities and the distance problem on Riemannian manifolds}
	\author{Alex Iosevich}
	\address{Department of Mathematics, University of Rochester, Rochester NY, 14620}
	\email{alex.iosevich@rochester.edu}

	\author{Bochen Liu}
	\address{Department of Mathematics \& International Center for Mathematics, Southern University of Science and Technology, Shenzhen 518055, China}
	\email{Bochen.Liu1989@gmail.com}

	\author{Yakun Xi}
	\address{School of Mathematical Sciences, Zhejiang University, Hangzhou 310027, PR China}
	\email{yakunxi@zju.edu.cn}

\begin{abstract}
	  We study the generalization of the Falconer distance problem to the Riemannian setting. In particular, we extend the result of Guth--Iosevich--Ou--Wang for the distance set in the plane to general Riemannian surfaces. Key new ingredients include a family of refined microlocal decoupling inequalities, which are related to the work of Beltran--Hickman--Sogge on Wolff-type inequalities, and an analog of Orponen's radial projection lemma which has proved quite useful in recent work on distance sets.
\end{abstract}
\maketitle
	
\section{Introduction}

The Falconer distance problem has been a central and persistently difficult question in harmonic analysis and geometric measure theory since 1986 when it was introduced by Falconer \cite{Falc86}. He conjectured that if the Hausdorff dimension of a compact subset $E \subset {\Bbb R}^d$, $d \ge 2$, is greater than $\frac{d}{2}$, then the Lebesgue measure of the distance set, 
$$ \Delta(E)=\{|x-y|: x,y \in E\}$$ is positive, where $|{}\cdot{}|$ denotes the Euclidean distance. Considering a suitably thickened and scaled integer lattice shows that the exponent $\frac{d}{2}$ would be best possible. 

Falconer \cite{Falc86} proved that the Lebesgue measure of $\Delta(E)$ is positive if $\dim_{{\mathcal H}}(E)>\frac{d+1}{2}$. This exponent was lowered to $\frac{13}{9}$ in two dimension by Bourgain \cite{B94} and to $\frac{4}{3}$ by Wolff \cite{W99}. Erdogan \cite{Erd05} established the threshold $\frac{d}{2}+\frac{1}{3}$ and the subject remained stuck for a while until a flurry of activity in the last  couple of years, culminating in the exponent $\frac{5}{4}$ by Guth, Iosevich, Ou and Wang \cite{GIOW19} in the plane, the exponent $1.8$ by Du, Guth, Ou, Wang, Wilson and Zhang \cite{DGOWWZ18} in $\R^3$, and the exponent $\frac{d^2}{2d-1}$ by Du and Zhang \cite{DZ18} in higher dimensions.  

It is interesting to note that the $\frac{5}{4}$ result in two dimensions is pinned in the sense that the authors prove that there exists $x \in E$ such that the Lebesgue measure of $\Delta_x(E)=\{|x-y|: y \in E\}$ is positive. The transition to pinned results was made possible, in part, due to a result by the second author \cite{Liu18} who established the $\frac{4}{3}$ pinned threshold in two dimensions. 

It is also interesting to formulate an analog of this problem on manifolds. Let $M$ be a $d$-dimensional compact Riemannian manifold without a boundary. Let $g$ be the associated Riemannian metric, $d_g$ the induced distance function, and for $E \subset M$ define 
$$\Delta_{g}(E)=\{d_g(x,y): x,y \in E \}.$$ Once again, we ask how large $dim_{{\mathcal H}}(E)$ needs to be to ensure that the Lebesgue measure of $\Delta_{g}(E)$ is positive. The Peres--Schlag machinery \cite{PS00} implies that if $dim_{{\mathcal H}}(E)>\frac{d+1}{2}$, then there exists $x \in E$ such that the Lebesgue measure of the pinned distance set $$\Delta_{g,x}(E)=\{d_g(x,y): y \in E \}$$ is positive. In fact Peres--Schlag considered very general maps, called generalized projections, that are not even necessarily metrics. Later this problem was studied by Eswarathasan, Iosevich and Taylor \cite{EIT11} (non-pinned version), Iosevich, Taylor and Uriarte-Tuero \cite{ITU16} (pinned version), via Fourier integral operators. Although for technical reasons their general setups look different, on Riemannian metrics their bounds are exactly the same as Peres--Schlag. Recently, by applying local smoothing estimates of Fourier integral operators, Iosevich and Liu \cite{IL19} improve the dimensional exponent on pins (the exceptional sets), while no dimensional exponent better than $\frac{d+1}{2}$ was obtained if pins are required to lie in the given set itself. The main result of this paper is the following. 

\begin{theorem} \label{maindistance} Let $(M, g)$ be a two-dimensional Riemannian manifold without a boundary, equipped with the Riemannian metric $g$. Let $E \subset M$ be of Hausdorff dimension $>\frac{5}{4}$. Then there exists $x \in E$ such that the Lebesgue measure of $\Delta_{g,x}(E)$ is positive. \end{theorem} 

It is not difficult to see that for any compact two-dimensional Riemannian manifold $M$ without a boundary and any $\epsilon>0$ there exists $E \subset M$ of Hausdorff dimension $1-\epsilon$ such that the Lebesgue measure of $\Delta_{g}(E)$ is zero. This is accomplished by putting a suitable thickened and scaled arithmetic progression on a sufficiently small piece of a geodesic curve. More precisely, one projects the one-dimensional version of the classical Falconer sharpness example (Theorem 2.4 in \cite{Falc86}) onto a small piece of the geodesic. However, the situation in higher dimensions is much more murky. We remind the reader that, in the plane, sharpness examples can be constructed from either arithmetic progressions or lattice points, while in higher dimensions only lattice points would do and the translation invariance of the Euclidean metric is crucial (Theorem 2.4 in \cite{Falc86}). As a general metric may not be translation invariant, now we only have the sharpness of $\frac{d}{2}$ when $d=2$, from ``arithmetic progressions" along a geodesic. We suspect that a generic $d$-dimensional Riemannian manifold possess a subset $E$ of Hausdorff dimension $\frac{d}{2}-\epsilon$, $\epsilon$ small, such that the Lebesgue measure of $\Delta_{g}(E)$ is zero. We shall endeavor to address this question in a sequel. 

The proof of Theorem \ref{maindistance}, just as the proof of its Euclidean predecessor (Theorem 1.1 in \cite{GIOW19}), is based on decoupling theory, a series of Fourier localized $L^p$ inequalities that underwent rapid development in recent years due to the efforts of Bourgain, Demeter, Guth and others. The application of decoupling theory to the proof of Theorem \ref{maindistance} has a variety of new features and complications stemming from the general setup of Riemannian manifolds. In particular, we shall prove a family of decoupling inequalities (Theorem \ref{decoupling-theorem}), which respect a certain microlocal decomposition that naturally generalizes the one used in the Euclidean decoupling theory of Bourgain and Demeter. Similar decompositions were used in the work of Blair and Sogge \cite{BS15} to study concentration of Laplace eigenfunctions. Our work on these variable coefficient decoupling inequalities is inspired by the work of Beltran--Hickman--Sogge \cite{BHS17}, where the authors proved certain Wolff-type decoupling inequalities for the variable coefficient wave equation and then used them to obtain sharp local smoothing estimates for the associated Fourier integral operators. Our proof of Theorem \ref{decoupling-theorem} uses the idea in \cite{BHS17} that one can exploit the multiplicative nature of the decoupling constant to make use of a gain at small scales.  We believe that, like the inequalities obtained in \cite{BHS17}, the microlocal decoupling inequalities we prove here are interesting in their own right.

This paper is structured as follows.  We motivate and set up our main decoupling inequalities in Section \ref{sectiondecoupling}, with the proof of the key decoupling results carried out in Section \ref{sectionproofkeytheorem}, parabolic rescaling in Section \ref{sectionparabolicrescaling} and the refined decoupling inequality in Section \ref{sectionrefined}. The application of the decoupling technology to distance sets on Riemannian manifolds is set up in Section \ref{sectiondecouplingtodistance} and carried out in Section \ref{L1bad} and \ref{sectionproofmain}. The key analog of Orponen's radial projection lemma used in \cite{GIOW19} is established in Section \ref{Sec-radial-proj}. 

{\bf Notation.} We shall write $A\lesssim B$, if there is an absolute constant $C$ such that $A\le CB$. We shall write $A\approx B$, if $\frac9{10} B\le A\le \frac{10}{9}B$. We shall use $\ra(R)$ to denote a term that is rapidly decaying in $R>1$, that is, for any $N\in\N$, there exists a constant $C_N$ such that $|\ra (R)|\le C_N R^{-N}$. We shall say that a function is essentially supported in a set $E$, if the $L^{\infty}$ norm of the tail outside $E$ is $\ra(R)$ for the underlying parameter $R$. 

{\bf Acknowledgment.} The authors would like to thank the anonymous referee for his/her thorough reading of this manuscript and many constructive comments. Liu was partially supported by the grant CUHK24300915 from the Hong Kong Research Grant Council, and a direct grant of research (4053341) from the Chinese University of Hong Kong. Xi was partially supported by the AMS-Simons travel grant and NSF China grant No. 12171424.
\vskip.125in 

\section{A microlocal decoupling inequality}
\label{sectiondecoupling}Like its Euclidean counterpart, the oscillatory integral operator $S_\lambda$, given by
\[S_\lambda f(x)=\int_{M} e^{i\lambda \phi(x,y)}\,a(x,y)\,f(y)\,dy,\quad a\in C_0^\infty(M\times M),\ \phi(x,y)=d_g(x,y),\]
 plays an important role in the study of the Falconer distance problem in the Riemannian setting.
In this section we introduce a decoupling inequality associated to a certain microlocal profile that is closely tied to the above operator. Indeed, we shall consider a more general class of oscillatory integral operator
\[S_\lambda f(x)=\int_{\mathbb R^d} e^{i\lambda \phi(x,y)}\,a(x,y)\,f(y)\,dy,\]
where $a\in C_0^\infty(\R^d\times\R^d)$ and the phase function $\phi(x,y)\in C^\infty(\supp\,a)$ satisfies the {\it Carleson--Sj\"olin condition} (see e.g. Corollary 2.2.3 in \cite{Sog17}). 

\begin{definition}[Carleson--Sj\"olin condition]\label{CS-condition}
	We say the phase function $\phi$ in $S_\lambda$ satisfies the $d$-dimensional Carleson--Sj\"olin condition if
	\begin{enumerate}
		\item for all $(x,y)\in\supp\,a(x,y)$,
		$$\text{rank}\left(\frac{\partial^2\phi}{\partial x\partial y}\right)=d-1;$$
		
		\item For all $x_0\in \supp_x a(x,y)$, $y_0\in \supp_y a(x,y)$, the Gaussian curvature of the $C^\infty$-hypersurfaces
		$$S_{x_0}=\{\nabla_x\phi(x_0,y):a(x_0,y)\neq 0\},\ S_{y_0}=\{\nabla_y\phi(x,y_0):a(x,y_0)\neq 0\}$$
		is positive and $\approx 1$ everywhere.
	\end{enumerate}
\end{definition}
\begin{remark} \label{pure} Denote $x=(x',x_d)$. For simplicity, we will only work with phase functions which are normalized in the sense that \begin{equation}\label{mix partials}\Big|\text{det}\left(\frac{\partial^2\phi}{\partial x'\partial y'}\right)\Big|\approx 1 \text{  on  }  \supp\, a(x,y),\end{equation} and all other entries in the mixed Hessian are small. In addition, we will also assume that \begin{equation}\label{pure partials}\Big|\text{det}\left(\frac{\partial^2\phi}{\partial x'^2}\right)\Big|,\  \Big|\text{det}\left(\frac{\partial^2\phi}{\partial y'^2}\right)\Big| \approx 1 \text{  on  }  \supp\, a(x,y).\end{equation} This can be guaranteed by adding terms purely in $x$ or $y$ to $
	\phi$, which will not change the $L^p$ norm of $f$ or $S_\lambda f$. These conditions are satisfied by the Riemannian distance function $d_g(x,y)$ if the points $x$ and $y$ are separated and positioned on the last coordinate axis.
\end{remark}
Notice that by Fourier inversion
$$S_\lambda f(x)=\frac{1}{(2\pi)^d}\int_{\mathbb R^d}\left(\iint_{\mathbb R^d\times\R^d} e^{i(\lambda \phi(z,y)-(z-x)\cdot\xi)}a(z,y)\,dz\,d\xi\right)f(y)\,dy.$$
Fix $y$ and consider the kernel. By integration by parts in the $z$ variable, one can see that the kernel has rapid decay in $\lambda$ if $|\lambda\nabla_z\phi(z,y)-\xi|>\lambda^\epsilon$. This implies, by losing a negligible error, it is enough to consider pairs $(z, \xi)$ such that $\xi$ lies in the $\lambda^{\epsilon}$-neighborhood of the rescaled hypersurface $\lambda S_z$. With this operator in mind, we will prove a

 general decoupling inequality associated to such a microlocal profile.

For convenience,  we set $\phi^\lambda(x,y):=\lambda\phi(x/\lambda, y/\lambda)$, $a^\lambda(x,y)=a(x/\lambda,y/\lambda)$, $S^\lambda_x=S_{x/\lambda}$ and work on the rescaled ball $B_\lambda$. Here, and throughout, $B_r^k(x)$ denotes the ball in $\R^k$ centered at $x$ of radius $r$. For convenience we let $B_r(x)=B^d_r(x)$. We also write $B_r$ if its center is not of particular interest.

Suppose $1\leq R\leq \lambda$. Denote
\[N_{\phi^\lambda,R}=\{(x,\xi): \text{dist}(\xi,S^\lambda_x)\le R^{-1}\}.\]
We choose a cutoff function $\psi_{\phi^\lambda,R}(x, \xi)\in C_0^\infty(T^*\mathbb R^d)$, which equals 1 on $N_{\phi^\lambda,R}$, and equals zero outside  $N_{\phi^\lambda,R/2}$.

For each $(x,\xi)\in N_{\phi^\lambda, R}$, we consider the curve
\[C_{(x,\xi)}=\{y\in B_\lambda:\nabla_x\phi^\lambda(x,y)=\xi\}.\]
We may parametrize this curve by $y(s)$ such that $|\partial_s y|=1$. Then $\nabla_x\phi^\lambda(x, y(s))=\xi$ and by differentiating both sides in $s$ we have
$$\frac{\partial^2\phi}{\partial x\partial y}\cdot\partial_s y(s)=0.$$
Since \eqref{mix partials} holds and all other entries in the mixed Hessian are assumed to be small, it follows that $|\partial_s y'(s)|$ is small. This implies $|\partial_s y_d(s)|\approx 1$, as $|\partial_s y|=1$. Therefore $C_{(x,\xi)}$ intersects the hyperplane $x_d=0$ at some \begin{equation}\label{intersect-x-axis}(u(x,\xi),0)\in\mathbb R^{d}.\end{equation}
Let
\begin{equation}\label{def-theta}\theta(x, \xi) = -\frac{\nabla_y \phi^\lambda(x,(u(x,\xi),0))}{|\nabla_y \phi^\lambda(x,(u(x,\xi),0))|}\in S^{d-1}.\end{equation} For example, in the Riemannian case where $\phi=d_g$, $C_{(x,\xi)}$ is a geodesic that has tangent vector $\xi$ at the point $x$, $(u(x,\xi),0)$ is the point where this geodesic intersects $x_d=0$, and $\theta(x, \xi)$ is a unit tangent vector of $C_{(x,\xi)}$ at $(u(x,\xi),0)$. Notice $u(x,\xi)$ is unique as long as we work within the injectivity radius of the manifold. More generally the local existence and uniqueness of the two functions $u(x,\xi)$ and $\theta(x,\xi)$ are guaranteed by \eqref{mix partials}, \eqref{pure partials} and the Carleson--Sj\"olin condition.

Now we cover $S^{d-1}$ using $R^{-1/2}$-caps $\tau$ and further decompose $N_{\phi^\lambda,R}$. Denote
\begin{equation}\label{partition}N^\tau_{\phi^\lambda,R}=\left\{(x,\xi)\in N_{\phi^\lambda,R}:\theta(x, \xi)\in \tau\right\}\end{equation}
and $\psi^\tau_{\phi^\lambda,R}(x,\xi)$ as a smooth partition of unity associated to this decomposition so that $\psi_{\phi^\lambda,R}(x,\xi)=\sum_\tau\psi^\tau_{\phi^\lambda,R}(x,\xi).$

We now define a function with microlocal support $N_{\phi^\lambda,R}$ and state the decoupling inequality under the decomposition $N^\tau_{\phi^\lambda,R}$.

\begin{definition}
	We say a smooth function $F$ has microlocal support in $N_{\phi^\lambda,R}$ if 
	$$F(x) = \iint_{\mathbb R^d\times\mathbb R^d} e^{-2\pi i(z-x)\cdot\xi}\, \psi_{\phi^\lambda,R}(z, \xi)\, F(z)\,dz\,d\xi + \ra(R)||F||_1.$$
\end{definition}

Denote
$$F_\tau(x) = \iint_{\mathbb R^d\times\mathbb R^d} e^{-2\pi i(z-x)\cdot\xi}\, \psi^\tau_{\phi^\lambda,R}(z, \xi)\, F(z)\,dz\,d\xi. $$

Throughout this paper, $\omega_{B_R}$ is a weight with Fourier support in $B_{1/R}(0)$ and satisfy
\begin{equation}\label{weight}\chi_{B_R}(x)\lesssim \omega_{B_R}(x) \lesssim \left(1+\frac{|x-c(B_R)|}{R}\right)^{-10d},\end{equation}
where $\chi_{B_R}$ is the indicator function of $B_R$, and $c(B_R)$ denotes the center of $B_R$.

\begin{theorem} [The main decoupling inequality] \label{decoupling-theorem} 
	Let $F$ be a function with microlocal support in $N_{\phi^\lambda,R}$. For any $\epsilon>0$, there exists a constant $C_\epsilon>0$ such that for any ball $B_R$ of radius $R$ contained in $B(0,\lambda)$, we have
	\begin{equation}\label{dec}\|F\|_{L^p(B_R)}\le C_\epsilon R^{\epsilon}\left(\sum_{\tau: R^{-1/2}-caps} \|F_\tau\|^2_{L^p(\omega_{B_R})}\right)^\frac12,\end{equation}
	for all $1\le R\le \lambda$, and $2\le p\le\frac{2(d+1)}{d-1}.$
\end{theorem}
\begin{remark}By Minkowski inequality and the freedom given by allowing the $\epsilon$ loss, it suffices to prove \eqref{dec} for $1\le R\le \lambda^{1-\epsilon/d}$. We have discussed about the microlocal support of $S_\lambda f$ at the beginning of this section, thus this decoupling inequality applies to $S_\lambda f$. 
 	 
\end{remark}

\section{Proof of Theorem \ref{decoupling-theorem}}
\label{sectionproofkeytheorem} 
\subsection{Small Scale Decoupling}
We first show that at a sufficiently small scale $R\ll\lambda$, \eqref{dec} can be reduced directly to Bourgain--Demeter. We shall then obtain the general case by induction on scales. Let $0<K\le \lambda^{1/2-\delta}$, where $0<\delta<\frac12$ is small but fixed. By the definition of $\omega_{B_K}$ in \eqref{weight}, we have
$$\sum_{\bar{x}\in K\Z}\omega_{B_K}(x-\bar{x})\approx 1$$
uniformly. Also let $B_K(\bar{x})$ denote the ball centered at $\bar{x}$ of radius $K$.

\begin{lemma}\label{reduce-to-BD}
	Let $0<\epsilon\ll 1$, $\epsilon^2<\delta<\frac12$. Let $F$ be a function supported on $B_R$ with microlocal support in $N_{\phi^\lambda,K}$, where $1\le K\le\lambda^{1/2-\delta}$.  Then there is a constant $C_\epsilon>0$ such that
	\begin{equation}\label{localdec}\|F\|_{L^p({B_{K}(\bar{x})})}\le C_\epsilon K^\epsilon\left(\sum_{\sigma: K^{-1/2}-caps} \|F_\sigma\|^2_{L^p(\omega_{{B_{K}}({\bar x})})}\right)^\frac12+\text{\ra}(R)||F||_{L^1(B_R)}\end{equation}
	holds for any $2\le p\le \frac{2(d+1)}{d-1}$uniformly over the class of phase function $\phi$, and the position of $B(\bar x, K)\subset B_R$.
\end{lemma}

We remark that for the purpose of our application, it is enough for $K$ to be as small as $\lambda^{\epsilon/2d}$, but Lemma \ref{reduce-to-BD} works with $K=\lambda^{1/2-\delta}$, for any $\epsilon^2<\delta<\frac12.$

\begin{proof}
	By definition,
	\begin{multline*}F(x)\,\chi_{B_K(\bar{x})}(x)=\chi_{B_K(\bar{x})}(x)\iint_{\mathbb R^d\times\mathbb R^d} e^{-2\pi i(z-x)\cdot\xi}\, \psi_{\phi^\lambda, K}(z, \xi)\, F(z)\,dz\,d\xi\\+\ra(R)||F||_{L^1(B_R)}\end{multline*}
	and
	$$F_\sigma(x)\,\omega_{B_K(\bar{x})}(x)= \omega_{B_K(\bar{x})}(x)\iint_{\mathbb R^d\times\mathbb R^d} e^{-2\pi i(z-x)\cdot\xi}\, \psi^\sigma_{\phi^\lambda, K}(z, \xi)\, F(z)\,dz\,d\xi.$$
	
	For each fixed $z$, both $\psi_{\phi^\lambda, K}(z, {}\cdot{})$, $\psi^\sigma_{\phi^\lambda, K}(z, {}\cdot{})$ can be written as sums of $\lesssim K^{10d}$ smooth functions supported on balls of radius $\approx K$ and equal $1$ on balls of radius $K/2$.	On each piece, integrating by parts in $\xi$ gives rapid decay in $\lambda$ unless $z$ is in the set
	$$\{z: |z-\bar x|\lesssim \lambda^{\epsilon^2}K\},$$
	which is still rapid decaying in $\lambda$ after adding $\lesssim K^{10d}$ terms together. It follows that
	\begin{multline*}F(x)\,\chi_{B_K(\bar{x})}(x)=\chi_{B_K(\bar{x})}(x)\int_{\mathbb R^d}\int_{B_{\lambda^{-\epsilon^{2}}K^{-1}}(\bar{x})} e^{-2\pi i(z-x)\cdot\xi}\, \psi_{\phi^\lambda, K}(z, \xi)\\ F(z)\,dz\,d\xi+\ra(R)||F||_{L^1(B_R)}\end{multline*}
	and
	\begin{multline*}F_\sigma(x)\,\omega_{B_K(\bar{x})}(x)= \omega_{B_K(\bar{x})}(x)\int_{\mathbb R^d}\int_{B_{\lambda^{-\epsilon^{2}}K^{-1}}(\bar{x})} e^{-2\pi i(z-x)\cdot\xi}\, \psi^\sigma_{\phi^\lambda, K}(z, \xi)\\ F(z)\,dz\,d\xi+\ra(R)||F||_{L^1(B_R)}.\end{multline*}

	\begin{observation}
		There exists an absolute constant $C>0$ such that:
		\begin{enumerate}
			\item $\bigcup_{z\in B_{\lambda^{\epsilon^2}K}(\bar{x})} N_{\phi^\lambda, K}(z, {}\cdot{})$ is contained in a $CK^{-1}$-neighborhood of $N_{\phi^\lambda, K}(\bar{x},{}\cdot{})$.
			\item for any $K^{-1/2}$-cap $\sigma$, $\bigcup_{z\in B_{\lambda^{\epsilon^2}K}(\bar{x})} N^\sigma_{\phi^\lambda, K}(z, {}\cdot{})$ is contained in a $CK^{-1}$-neighborhood of $N^\sigma_{\phi^\lambda, K}(\bar{x},{}\cdot{})$.
		\end{enumerate}
	\end{observation}
	\begin{figure}\begin{center}
			\includegraphics[width =13cm]{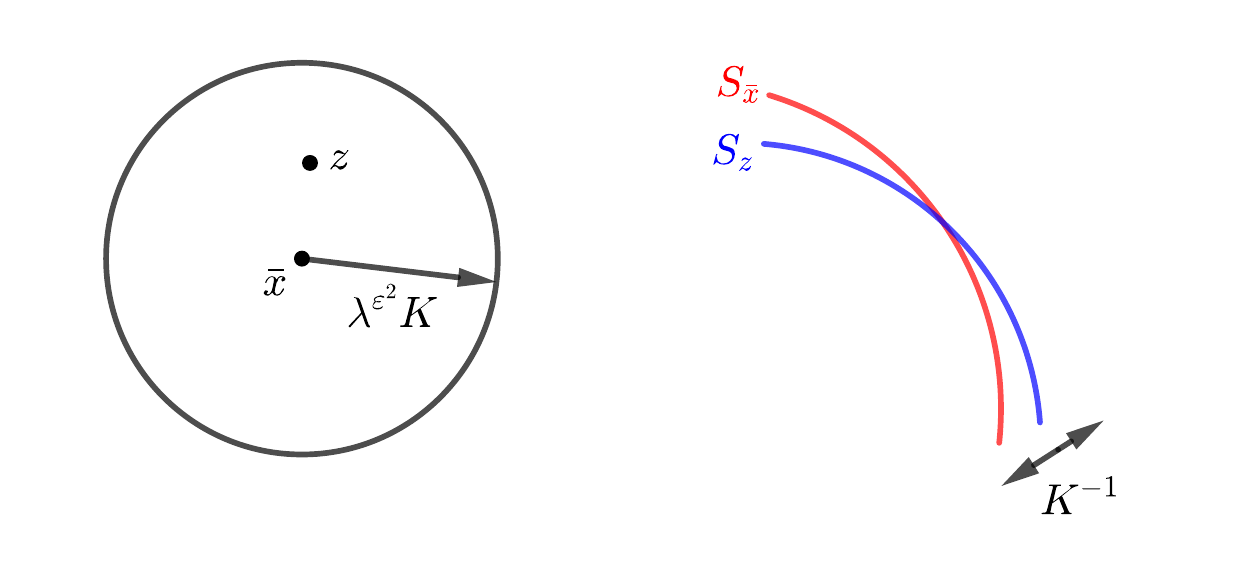}
			\caption{$S_z$ and $S_{\bar x}$.}
			\label{small scale}
		\end{center}
	\end{figure}

	This observation implies that the Fourier support of $F$ lies in a $CK^{-1}$-neighborhood of $S^\lambda_{\bar x}$, and the Fourier supports of $F_\sigma$ are $CK^{-1}$-caps with finite overlap. Then Lemma \ref{reduce-to-BD} follows directly from Bourgain--Demeter's decoupling inequality \cite[Theorem 1.1]{BD15}. Now it remains to prove the observation.
	
	\proof[Proof of Observation.] See Figure \ref{small scale}. The claim (1) follows from the fact that
		$$|\nabla_x\phi^\lambda(z, y)-\nabla_x\phi^\lambda(\bar{x}, y)|\lesssim|z-\bar{x}|/\lambda\lesssim \lambda^{\epsilon^2}K/\lambda\le\lambda^{-1/2-\delta+\epsilon^2}\lesssim K^{-1}.$$
		To see (2), since the function $\theta$ defined in \eqref{def-theta} is differentiable,
		$$|\theta(z/\lambda, \xi)- \theta(\bar{x}/\lambda, \xi)|\lesssim \lambda^{\epsilon^2}K/|\lambda|\lesssim K^{-1}.$$ Hence for each $\xi\in N^\sigma_{\phi^\lambda, K}(z, \cdot)$, $\theta(\bar{x}, \xi)$ lies in a $CK^{-1}$-neighborhood of $\sigma$.

\end{proof}

Now if we cover $B_R$ using balls of radius $K$ and invoke Minkowski inequality to swap the order of the summations, we obtain the following estimate on $B_R$. The term $\ra(R)||F||_{L^1(B_R)}$ is absorbed by the left hand side.
\begin{lemma}[Small Scale Decoupling]\label{small-scale}Given $0<\epsilon\ll 1$, $\epsilon^2<\delta<\frac12$ and $1\le K\le\lambda^{1/2-\delta}$, $0<\epsilon\ll 1$. Let $F$ be a function with microlocal support in $N_{\phi^\lambda,K}$. There exists $C_\epsilon>0$ such that
	\begin{equation}\label{smallscale}\|F\|_{L^p(B_R)}\le C_\epsilon K^\epsilon(\sum_{\sigma: K^{-1/2}-caps} \|F_\sigma\|^2_{L^p(\omega_{B_R})})^\frac12,\end{equation}
	for all $K\le R\le \lambda,$ $2\le p\le \frac{2(d+1)}{d-1}.$
\end{lemma}

\subsection{Iteration.}Now we are ready to prove Theorem \ref{decoupling-theorem}. Denote $D_\epsilon(R,\lambda)$ the smallest constant such that \eqref{dec} holds. The proof goes by induction. The base case of the induction follows from Lemma \ref{reduce-to-BD}. Assume that 
$$D_\epsilon(R',\lambda')\le \bar C_\epsilon$$
for all $0<R'<R/2$ and $R'\le\lambda'^{1-\epsilon/d}$.

To proceed with our iteration, we shall use the following parabolic rescaling lemma which will be proved in the next section. 
\begin{lemma}\label{para}Let $1\le\rho\le R^\frac12\le\lambda^\frac12$ and $\sigma$ be a $\rho^{-1}$-cap. Suppose $h$ has microlocal support contained in $N^\sigma_{\phi^\lambda,R}$. Then for any $\epsilon>0$, given $R\leq\lambda^{1-\epsilon/d}$, $1\leq \rho\leq \lambda^{\epsilon/4d}$, there exists a constant $C_\epsilon'$ which is uniform among the class of $\phi$, such that
	\[\| h\|_{L^p(B_{R})}\le C'_\epsilon D_\epsilon(R/\bar C\rho^2,\lambda/\bar C\rho^2)(R/\rho^2)^{\epsilon} \Big(\sum_{\tau: {R^{-1/2}-caps}}\| h_\tau\|^2_{L^p(\omega_{B_{R}})}\Big)^\frac12.\]
\end{lemma}

Applying Lemma \ref{para} to \eqref{smallscale} with $\rho^2=K=\lambda^{\epsilon/2d}$. It follows that
\[\| F\|_{L^p(B_{R})}\le C'_\epsilon D_\epsilon(R/\bar CK,\lambda/\bar CK)(R/K)^{\epsilon} C_\epsilon K^{\epsilon/2} \Big(\sum_{\tau: {R^{-1/2}-caps}}\| F_\tau\|^2_{L^p(\omega_{B_{R}})}\Big)^\frac12.\]

By our induction hypothesis, $D_\epsilon(R/\bar CK,\lambda/\bar CK)\le \bar C_\epsilon$. Therefore
\[\| F\|_{L^p(B_{R})}\le C'_\epsilon \bar C_\epsilon R^{\epsilon} C_\epsilon K^{-\epsilon/2} \Big(\sum_{\tau: {R^{-1/2}-caps}}\| F_\tau\|^2_{L^p(\omega_{B_{R}})}\Big)^\frac12.\]
If we choose $K$ large enough such that $C'_\epsilon C_\epsilon K^{-\epsilon/2}\le 1$, then the induction closes.

\section{Parabolic rescaling and proof of Lemma \ref{para}}
\label{sectionparabolicrescaling} 
Parabolic rescaling is already a standard technique in harmonic analysis. Denote by $N_R$ the $R^{-1}$-neighborhood of a  cap on the hypersurface
$$\{(\omega, \Sigma(\omega)): \omega\in B^{d-1}_{1/\rho}\},$$
parametrized by $\Sigma\in C^\infty(\mathbb R^{d-1})$. Here we assume that $\Sigma(0)=0,$ $\nabla \Sigma(0)=0$, and the curvature of the hypersurface is $\approx 1$. The key idea is that 
$$\{(\rho u, \rho^2 t): (u,t)\in N_R\} $$
is the $(R/\rho^2)^{-1}$-neighborhood of a cap whose curvature is also $\approx 1$.  Although our version of the parabolic rescaling lemma is more complicated, this geometric fact will be still used as a key ingredient in the proof. Using parabolic rescaling, we now prove Lemma \ref{para}.

Let $\theta_0$ be the center of the $(R/\rho)^{-1}$-cap $\sigma\in S^{d-1}$. Without loss of generality, assume that $B_R$ is centered at the origin. Consider a finite overlapping collection of curved tubes with length $R$ and cross-section radius $R\rho^{-1}$ that are positioned with respect to $\sigma$ and cover $B_R$.  More precisely, decompose $[- R ,  R )^{d-1}$ into 
$$I=[j_1 R /\rho,(j_1+1) R /\rho)\times\cdots\times [j_{d-1} R /\rho,(j_{d-1}+1) R /\rho), \ j_i\in[-\rho,\rho]\cap\Z$$ 
and then decompose $B_ R $ into
$$T=\{x\in B_ R : u(x, \theta_0)\in I\},$$
where $u(x, \xi)$ is as defined in \eqref{intersect-x-axis}. It follows that
\[\|h\|^p_{L^p(B_{R})}	\sim \sum_{T}\| h\|^p_{L^p(T)}.\]

Therefore, if we can decouple $h$ on each $T$, we can decouple $h$ over the whole ball by invoking Minkowski inequality. Thus, from now on, we assume that $h$ is localized to such a tube, that is, we shall abuse our notation a bit and consider the function 
\begin{equation}\label{localizing h}
h=h|_T.
\end{equation}

Fixing a tube $T$, after rotation and translation, we may assume the center of $\sigma$ is $\theta_0=(\vec{0}, 1)$, $\vec{0}\in\mathbb R^{d-1}$, and the central curve passes through the origin. Moreover, in view of \eqref{def-theta}, we also assume that a point $x$ on the central curve satisfies $$\nabla_{y'}\phi (x, 0)=\vec 0.$$

\subsection{Straightening the Tube.}
We would like to perform a parabolic rescaling over the tube $T$. However, since our tube $T$ is defined to be the tubular neighborhood of the central curve, the tangent to the central curve of the tube is changing smoothly along the curve. Thus a simple stretching along fixed orthogonal directions will not always provide desired rescaling for the caps along the curve.
To fix this problem, we would like to do a smooth change of variables to Lagrangian coordinates before we do parabolic rescaling. See Figure \ref{parab}. \begin{figure}\begin{center}
		\includegraphics[width =13cm]{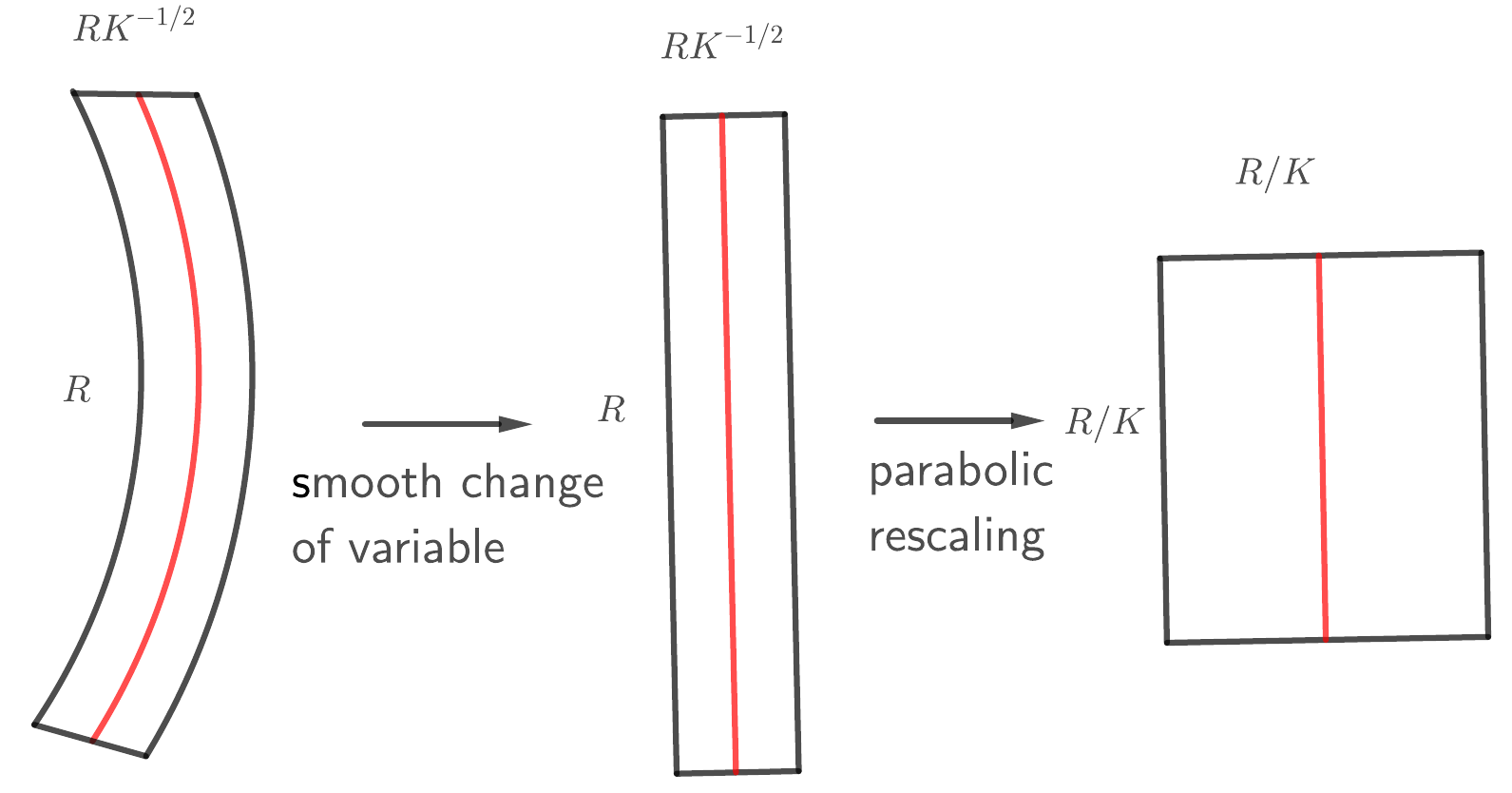}
		\caption{Straightening the tube and parabolic rescaling.}
		\label{parab}
	\end{center}
\end{figure}

To describe this change of variable, first we rescale back to the unit ball, where $T$ get rescaled back to a tube of length $ R\lambda^{-1}$

 cross-section radius $R(\lambda\rho)^{-1}$. Recall that a point $x$ on the central curve must satisfy $$\nabla_{y'}\phi (x, 0)=\vec 0.$$

Denote
$$F(x, u)=\nabla_{y'}\phi (x, (u, 0))\in\R^{d-1}. $$
Since the center of $\sigma$ is $(\vec{0},1)$ and the central curve passes through the origin, we have $F(x, 0)=\vec 0$ for each $x$ in the central curve of $T$. Also by \eqref{mix partials}
$$|\det(\nabla_{x'} F)|=\left|\det\left(\frac{\partial^2\phi}{\partial x'\partial y'}\right)\right|\approx 1,$$
so by the implicit function theorem for any $(x', t)$ in the central curve there exists a neighborhood $U_t\subset\mathbb R^{d-1}$ of $x'$ and a differentiable function $\alpha({}\cdot{}, t)$ on this neighborhood such that for any $u\in U_t$,
\begin{equation}\label{implicit-function-theorem-alpha}\nabla_{y'}\phi ((\alpha(u, t), t), (u, 0))=0.\end{equation}
By continuity there exists a domain $U\subset\R^d$ such that \eqref{implicit-function-theorem-alpha} holds whenever $(u,t)\in U$. Notice $\alpha(u, t)$ is also differentiable in $t$ as $\{(\alpha(u, t), t):t\in(-R/\lambda,R/\lambda)\}$ is a differentiable curve by the regular value theorem.

Then denote
$$G(v,(w,s))=\nabla_{y'}\phi((\alpha(v,0),0), (w,s)). $$
Since the center of $\sigma$ is $(\vec{0},1)$ and the central curve starting from the origin passes through $(\alpha(0,0),0)$, we have $G(0, (0,0))=0$. Also by \eqref{pure partials}
$$|\det(\nabla_{w} G)|=\left|\det\left(\frac{\partial^2\phi}{\partial y'^2}\right)\right|\approx 1.$$
Similar to the existence of $U$, there exists a neighborhood $V\subset\mathbb R^d$ and a differentiable function $\beta$ on this neighborhood such that
\begin{equation}\label{implicit-function-theorem-beta}\nabla_{y'}\phi ((\alpha(v, 0), 0), (\beta(v,s), s))=0.\end{equation}

Notice the size of $U, V$ is independent in $\rho, R, \lambda$, thus we may assume $(x,y)\in U\times V$ for any pair $(x,y)$ under consideration. Also from \eqref{implicit-function-theorem-alpha} and \eqref{implicit-function-theorem-beta}, we have
\begin{equation}\label{initial-point}\beta(u,0)=u.\end{equation}

Change variables from $x=(\alpha(u,t), t)$, $y=(\beta(v,s), s)$ to $(u,t), (v,s)$. We claim that the associated Jacobians are $\approx 1$. In fact, if we take $\nabla_u$ on both sides of \eqref{implicit-function-theorem-alpha}, $\nabla_v$ on both sides of \eqref{implicit-function-theorem-beta}, we have
$$\frac{\partial^2\phi}{\partial x'\partial y'}\cdot \nabla_u \alpha+\frac{\partial^2\phi}{\partial y'^2} =0,\  \frac{\partial^2\phi}{\partial x'\partial y'}\cdot \nabla_u \alpha+\frac{\partial^2\phi}{\partial y'^2}\cdot\nabla_v\beta =0.$$

Since $\left|\det\left(\frac{\partial^2\phi}{\partial y'^2}\right)\right|, \left|\det\left(\frac{\partial^2\phi}{\partial x'\partial y'}\right)\right|\approx 1$, we have \begin{equation}\label{alpha bound}|\det(\nabla_u \alpha)|\approx 1,\end{equation} and then \begin{equation}\label{beta bound}|\det(\nabla_v \beta)|\approx 1.\end{equation} Hence the Jacobians  are $$|\det(\nabla_u \alpha)|\cdot \frac{\partial t}{\partial t}\approx 1,$$ and $$ |\det(\nabla_v \beta)|\cdot\frac{\partial s}{\partial s} \approx 1.$$

Now we restrict our attention to $T$. Firstly, note that $T$ is the $(R/\lambda\rho)$-neighborhood of the central curve 
$$\{(\alpha(0, t), t):t\in(-R/\lambda,R/\lambda)\}.\ $$
Secondly, if we denote
\begin{equation}\label{Nx}N^{\sigma}_{\phi^\lambda, R}(x):=\{\xi: (x,\xi)\in N^{\sigma}_{\phi^\lambda, R}\},\end{equation} we shall prove in the next subsection that the $(R/\rho)^{-1}$-cap $N^{\sigma}_{\phi^\lambda, R}(\alpha(0,t),t)$ associated to $T$ is contained in the $R^{-1}$ neighborhood of the set
\begin{equation}\{\eta=\nabla_x\phi((\alpha(0,t),t), (\beta(v,s),s)):(v,s)\in B_{R/\lambda\rho}^{d-1}\times (-R/\lambda,R/\lambda)\}\subset S_{(\alpha(0,t),t)}.\end{equation}
It follows that $(u,t), (v,s)\in B_{R/\lambda\rho}^{d-1}\times (-R/\lambda,R/\lambda)$. 

\subsection{The associated cap}
Later in Section \ref{subsection_parabolic_rescaling} we shall apply the parabolic rescaling
$$(u,t) = (\rho \tilde u, \rho^2 \tilde t),\ \tilde{\eta} = (\rho \eta', \rho^2 \eta_d).$$
However, to apply the inductive hypothesis, it is necessary that the rescaling on $\eta'$ acts on the tangents of the associated cap at the center, while the rescaling on $\eta_d$ is along its normal at the center. Therefore it is crucial that the velocity of the central curve is a normal of the associated cap at the center. In this subsection we check this property.

Recall our tube $T$ is  the $(R/\lambda\rho)$-neighborhood of the central curve
\begin{equation}\label{central-curve}\{(\alpha(0,t),t)\in U:|t|<R/\lambda\}.\end{equation}
 We now show that the cap   $N^{\sigma}_{\phi^\lambda, R}(\alpha(0,t),t)$ can be parametrized by
\begin{equation}\label{associated_cap_of_T}\{\eta=\nabla_x\phi((\alpha(0,t),t), (\beta(v,s),s)):(v,s)\in B_{R/\lambda\rho}^{d-1}\times (-R/\lambda,R/\lambda)\}.\end{equation}
In other words we shall show that the velocity vector of \eqref{central-curve} is a normal of the cap \eqref{associated_cap_of_T} at the center.

Since the curve starts from the origin with initial velocity $(\vec{0},1)$, the center of the cap \eqref{associated_cap_of_T} is
\begin{equation}\label{center-S}\nabla_x \phi((\alpha(0,t), t), (0,0)).\end{equation}

Take $\partial_t$ on both sides of \eqref{implicit-function-theorem-alpha},
\begin{equation}\label{perpendicular}\frac{\partial^2\phi}{\partial x\partial y'}((\alpha(0, t), t), (0, 0))\cdot (\partial_t \alpha(0,t), 1)=0.\end{equation}

Notice that the vectors $$\partial_{y_j}\nabla_x\phi((\alpha(0,t), t), (0,0)),\ j=1,\dots,d-1,$$ generate the tangent space of the cap \eqref{associated_cap_of_T} at its center $\nabla_x\phi((\alpha(0,t), t), (0,0))$. Hence we conclude that the velocity vector of the curve $(\alpha(0,t),t)$, is a normal of \eqref{associated_cap_of_T} at its center.

Denote $$\tilde \phi ((u,t), (v, s))=\phi((\alpha(u,t),t), (\beta(v,s),s))$$
and rescale $(u,t)$, $(v,s)$ from $B_{R/\lambda\rho}^{d-1}\times (-R/\lambda,R/\lambda)$ to $B_{R/\rho}^{d-1}\times(-R, R)$: $$\tilde \phi^\lambda ((u,t), (v, s))=\lambda\,\tilde \phi ((u/\lambda,t/\lambda), (v/\lambda, s/\lambda)).$$ 
Also denote 
$$\alpha^\lambda(u,t)=\lambda\alpha(u/\lambda,t/\lambda),\ \beta^\lambda(u,t)=\lambda\beta(v/\lambda,s/\lambda).$$

Then one can define $\tilde S^\lambda_{(u,t)}$, $N_{\tilde \phi^\lambda, R}$, $N^{\sigma}_{\tilde \phi^\lambda, R}$ in terms of $\tilde{\phi}$.
Similar to \eqref{Nx}, we denote
$$N^{\sigma}_{\tilde \phi^\lambda, R}(u,t):=\{\xi: ((u,t),\xi)\in N^{\sigma}_{\tilde \phi^\lambda, R}\}.$$

We claim that $N^{\sigma}_{\tilde \phi^\lambda, R}(0,t)$ has normal $(\vec{0},1)$ at its center.  As $(\vec{0},1)$ is the velocity vector of the central curve of $T$ under new variables, we are ready for parabolic rescaling.

Now we prove this claim. Since
\begin{equation}\label{nabla-tilde-phi}\nabla_{(u,t)}\tilde \phi^\lambda = (\nabla_{x'}\phi^\lambda\cdot\partial_u\alpha^\lambda, \nabla_{x}\phi^\lambda\cdot(\partial_t\alpha^\lambda, 1))=\begin{pmatrix}
\nabla_u\alpha^\lambda & 0\\\partial_t\alpha^\lambda & 1
\end{pmatrix}\cdot \nabla^\lambda_x\phi,\end{equation}
it follows that
\begin{equation}\label{tilde-S}N^{\sigma}_{\tilde \phi^\lambda, R}(0,t) =\begin{pmatrix}
\nabla_u\alpha^\lambda & 0\\\partial_t\alpha^\lambda & 1
\end{pmatrix} N^{\sigma}_{\phi^\lambda, R}(\alpha(0,t),t)\end{equation}
and by \eqref{center-S} the center of $N^{\sigma}_{\tilde \phi^\lambda, R}(0,t)$ is
\begin{equation}\label{center-tilde-S}\nabla_{(u,t)}\tilde\phi^\lambda ((0,t), (0,0)). \end{equation}

For the normal at the center of $N^{\sigma}_{\tilde \phi^\lambda, R}(0,t)$, since $\alpha^\lambda$ in \eqref{nabla-tilde-phi} is independent in $v$, we can 
take $\nabla_{v}$ on the last component of \eqref{nabla-tilde-phi} to obtain
$$\partial_v\nabla_{(u,t)}\tilde \phi^\lambda=(\nabla_{v}\beta^\lambda(v,s))^t\cdot \frac{\partial^2\phi^\lambda}{\partial y'\partial x} ((0,t),(v,s)) \cdot (\partial_t\alpha^\lambda(0,t), 1)$$
which is $0$ when $(v,s)=(0,0)$ by \eqref{perpendicular}.

As $N^{\sigma}_{\tilde \phi^\lambda, R}(0,t)$ can be parametrized by $v$, this means $N^{\sigma}_{\tilde \phi^\lambda, R}(0,t)$ has normal $(\vec{0}, 1)$ at $(v,s)=(0,0)$, as desired. Also from \eqref{tilde-S} and \eqref{alpha bound} we have that the curvature of $N^{\sigma}_{\tilde \phi^\lambda, R}(0,t)$ is $\approx 1$.

Lastly, we remark that without loss of generality, we may assume the center of $N^{\sigma}_{\tilde \phi^\lambda, R}(0,t)$ is the origin. To see this, notice that we are allowed to change the phase function from $\tilde \phi^\lambda$ to
$$\tilde \phi^\lambda((u,t), (v,s)) - \tilde\phi^\lambda ((0,t), (0,0))=:\Phi^\lambda ((u,t), (v,s)).$$
Then by \eqref{initial-point}, \eqref{implicit-function-theorem-alpha} we have
$$\nabla_{(u,t)} \Phi^\lambda ((u,t), (v,s)) = \nabla_{(u,t)}\tilde \phi^\lambda((u,t), (v,s)) - \nabla_{(u,t)}\tilde\phi^\lambda ((0,t), (0,0)),$$
where the second term is nothing but the center of $N^{\sigma}_{\tilde \phi^\lambda, R}(0,t)$.

\subsection{Checking Microlocal Profile.}
Recall by \eqref{localizing h}, we have restricted $h$ to the tube $T$. Now if we take $\tilde h (u,t)=h(\alpha^\lambda(u,t), t)$, then modulo a $\ra(R)\|h\|_1$ error, its Fourier transform is:
\begin{equation}\label{Fourier-transform-tilde-h}\begin{aligned}\widehat{\tilde h} (\eta) &= \int_{\R}\int_{\R^{d-1}} e^{-2\pi i (w,r)\cdot\eta}h(\alpha^\lambda(w,r),r)\,dw\,dr\\&=\iiiint e^{-2\pi i ((w,r)\cdot\eta- ((\alpha^\lambda(w,r),r)-z)\cdot\xi)} \psi^\sigma_{\phi^\lambda, R}(z, \xi) h(z)\,d\xi\,dz\,dw\,dr\end{aligned}\end{equation}

Change variables $z=(\alpha^\lambda(u,t),t)$, then $h(z)=\tilde{h}(u,t)$. Fix $(u,t)$ and let
\[\mathcal N=\{\eta: |\eta'-\nabla_w\alpha^\lambda\cdot\xi'|\lesssim R^{-\frac12},|\eta_d-(\partial_r\alpha^\lambda\cdot\xi_d+\xi_d)|\lesssim R^{-1},(\alpha^\lambda(u,t),t),\xi)\in N^\sigma_{\phi^\lambda,R}\},\]
which is essentially the linear transformation of $N^\sigma_{\phi^\lambda,R}(\alpha^\lambda(w,r),r)$ via the matrix
\[\begin{pmatrix}\nabla_w \alpha^\lambda & 0 \\ \partial_r\alpha^\lambda &1.\end{pmatrix}\]
Let $\chi_\mathcal N((u,t),\eta)$ be a smooth cutoff function that equals $1$ on $\mathcal N$. 

If we run integration by parts in $(w,r)$, we get a rapid decay in $R$ unless
$$|\eta'-\nabla_w\alpha^\lambda\cdot\xi'|\lesssim \rho/R\le R^{-\frac12}$$
and
$$|\eta_d-(\partial_r\alpha^\lambda\cdot\xi_d+\xi_d)|\lesssim R^{-1}. $$
Therefore, at the expense of a $\ra(R)$ term, one can insert the cutoff $\chi_\mathcal N((u,t),\eta)$ into our integral \eqref{Fourier-transform-tilde-h}. Then, a similar argument shows that one can drop the cutoff $\psi^\sigma_{\phi^\lambda, R}(z, \xi)$ by losing another  $\ra(R)$ term. Therefore, modulo a $\ra(R)||\tilde{h}||_{L^1}$ error, we have
\begin{equation}\begin{aligned}\widehat{\tilde h} (\eta) &=\int\cdots\int e^{-2\pi i ((w,r)\cdot\eta- ((\alpha^\lambda(w,r),r)-(\alpha^\lambda(u,t),t))\cdot\xi)}\chi_\mathcal N((u,t),\eta) \tilde{h}(u,t)\,d\xi\,du\,dt\,dw\,dr\\&= \int_{\R}\int_{\R^{d-1}} e^{-2\pi i (w,r)\cdot\eta}\chi_\mathcal N((\omega,r),\eta)\,\tilde{h}(\omega,r)\,dw\,dr\end{aligned}\end{equation}

Recall we have pointed out that
$$\mathcal{N}=\begin{pmatrix}\nabla_w \alpha^\lambda & 0 \\ \partial_r\alpha^\lambda &1\end{pmatrix}N^\sigma_{\phi^\lambda,R}(\alpha^\lambda(w,r),r),$$
which is exactly $N^\sigma_{\tilde{\phi}^\lambda,R}(\omega,r)$ by the analogous of \eqref{tilde-S} in $T$. Therefore
$$\chi_{\mathcal N}((\omega,r),\eta)= \psi^{\sigma}_{\tilde \phi^\lambda, R} ((w,r), \eta),$$
and  modulo a $\ra(R)||\tilde{h}||_{L^1}$ error, we have
\begin{equation}\label{microlocal-support-tilde-h}\tilde h (u,t) = \int_{\R^d}\iint e^{-2\pi i ((w,r)-(u,t))\cdot\eta} \psi^{\sigma}_{\tilde \phi^\lambda, R} ((w,r), \eta)\,\tilde h (w,r)\,dw\,dr\,d\eta,\end{equation}
where $(u,t), (w,r)\in B_{R/\rho}^{d-1}\times (-R, R)$.

\subsection{Parabolic Rescaling}\label{subsection_parabolic_rescaling} The parabolic rescaling argument is a standard tool in problems related to Fourier restriction theory. See for reference the applications of parabolic rescaling in the  decoupling theory in the works of Bourgain and Demeter \cite{BD15}. The argument that we use here is a variant of the argument used by Beltran, Hickman and Sogge \cite{BHS17}.
Here we apply the parabolic rescaling $(u,t) = (\rho \tilde u, \rho^2 \tilde t)$, $(w,r) = (\rho \tilde w, \rho^2 \tilde r)$, $\tilde \eta = (\rho\eta', \rho^2\eta_d)$, and denote $$\tilde h^\rho(\tilde u,\tilde t)=\tilde h (\rho\tilde u,\rho^2\tilde t).$$ Then $(\tilde u, \tilde t), (\tilde w, \tilde r)\in B_{R/\rho^2}$ and \eqref{microlocal-support-tilde-h} is equivalent to
\begin{equation}\label{microlocal-support-reduction}\tilde h^\rho(\tilde u,\tilde t)=\int_{\R^d}\int_{\R}\int_{\R^{d-1}} e^{-2\pi i((\tilde w,\tilde r)-(\tilde u,\tilde t))\cdot\tilde \eta}\, \psi^{\tilde\sigma}_{\tilde \phi^\lambda,R}((w, r), \eta)\, \tilde h^\rho(\tilde w,\tilde r)\,d\tilde w\,d\tilde r\,d\tilde \eta.\end{equation}

By the geometric observation we mentioned at the beginning of this section, it follows that
$$\psi^{\tilde \sigma}_{\tilde \phi^\lambda,R}((w,r), \eta) \sim \psi_{(\tilde  \phi_\rho)^{\lambda/\rho^2},R/\rho^2}((\tilde w,\tilde r), \tilde \eta), $$
where the function $\tilde  \phi_\rho((\tilde u,\tilde t), (\tilde v,\tilde s))$ is chosen such that 
$$(\tilde  \phi_\rho)^{\lambda/\rho^2}((\tilde u,\tilde t), (\tilde v,\tilde s)) = \tilde \phi^\lambda((\rho\tilde u,\rho^2\tilde t), (\rho\tilde v, \rho^2\tilde s)).$$

Hence  \eqref{microlocal-support-reduction} becomes
$$\tilde h^\rho(\tilde u,\tilde t)=\int_{\R^d}\int_{\R}\int_{\R^{d-1}} e^{-2\pi i((\tilde v,\tilde s)-(\tilde u,\tilde t))\cdot\tilde \eta}\, \psi_{(\tilde \phi_\rho)^{\lambda/\rho^2},R/\rho^2}((\tilde v,\tilde s), \tilde \eta)\, \tilde h^\rho(\tilde v,\tilde s)\,d\tilde v\,d\tilde s\,d\tilde \eta.$$

Now we can apply our induction hypothesis to $\tilde h^\rho$ at scale $(R/\rho,\lambda/\rho)$ to obtain 
\[\| \tilde h^\rho\|_{L^p(\omega_{B_{R/\rho^2}})}\le C'_\epsilon D_\epsilon(R/\rho^2,\lambda/\rho^2)(R/\rho^2)^{\epsilon} (\sum_{\bar \tau: {\rho R^{-1/2}-caps}}\| (\tilde h^\rho)_{\bar \tau}\|^2_{L^p(\omega_{B_{R/\rho^2}})})^\frac12.\]

By our reduction above,
$$\|  h\|_{L^p(\omega_{T})}\approx \rho^{-(d+1)/p}\| \tilde h^\rho\|_{L^p(\omega_{B_{R/\rho^2}})}$$
and similarly for each $R^{-1/2}$-cap $\tau$, there is a $\rho R^{-1/2}$-cap $\bar \tau$ such that
$$\|  h_\tau\|_{L^p(\omega_{T})}\approx \rho^{-(d+1)/p}\|  (\tilde h^\rho)_{\bar\tau}\|_{L^p(\omega_{B_{R/\rho^2}})}.$$

Hence
\[\|  h\|_{L^p(\omega_{T})}\le C'_\epsilon D_\epsilon(R/\rho^2,\lambda/\rho^2)(R/\rho^2)^{\epsilon} (\sum_{\tau: { R^{-1/2}-caps}}\| h_\tau\|^2_{L^p(\omega_{T})})^\frac12,\]
and the proof is complete.

\section{A Refined Microlocal decoupling Inequality}
\label{sectionrefined} 
In this section, we shall use the microlocal decoupling inequalities proved above and the approach in the Euclidean case (\cite{GIOW19}) to obtain further refined decoupling inequalities. In this section, we shall only present the theorem in the two-dimensional case for the critical index $p=6$  for the sake of simplicity, even though the result can be readily generalized to higher dimensions case for all $2\le p\le2(d+1)/(d-1)$ with the same proof. Given $1\le R\le \lambda$, $\delta>0$, we shall consider  collections of curved tubes $\mathbb T=\cup_\tau \mathbb T_\tau$ associated to the phase function $\phi^{\lambda}$, and arcs $\tau\subset S^1$ of length $R^{-1/2}$. A curved tube $T$ of length $R$ and cross-section radius $R^{\frac12+\delta}$ is in $\mathbb T_\tau$, if the central axis of $T$ is the curve $\gamma_\tau^t$ given by
$$\gamma_\tau^t=\left\{x:-\frac{\nabla_y\phi^\lambda(x,(t,0))}{|\nabla_y\phi^\lambda(x,(t,0))|}{=\text{center of  }\tau}\right\},$$
here $t$ is chosen from a maximal $R^{\frac12+\delta}$ separated subset of $\mathbb R$, and $\theta$ is the center of $\tau\subset S^1$. 
For instance, in the case $\phi=d_g$, the angle between central geodesic $\gamma_\tau$ of $T$ and the first coordinate axis is  $\theta=$\,center of $\tau$. We say a function $f$ is microlocalized to a tube $T\in\mathbb T_\tau$ if $f$ has microlocal support essentially in $N^\tau_{\phi^\lambda,R}$, and $f$ is essentially supported in $T$ in physical space as well. We remark that the tubes considered in this section is different in size compared to those in the previous section.
\begin{theorem}\label{refined-decoupling-theorem}Let $1\le R\le\lambda$.
	Let $f$ be a function with microlocal support in $N_{\phi^\lambda,R}$. Let $\mathbb T$ be a collection of tubes associated to $R^{-1/2}$ caps, and $\mathbb W\subset \mathbb T$. Suppose that each $T\in\mathbb W$ is contained in the ball $B_R$. Let $W$ be the cardinality of $\mathbb W$. Suppose that 
	\[f=\sum_{T\in\mathbb W}f_T,\]
	where each $f_T$ is microlocalized to $T\in\mathbb W$, and $\|f_T\|_{L^6(\omega_{B_R})}$ is roughly constant among all $T\in\mathbb W$. If $Y$ is a union of $R^\frac12$-cubes in $B_R$ each of which intersects at most $M$ tubes $T\in\mathbb W$. Then
	\begin{equation}\label{rdec}\|f\|_{L^6(Y)}\le C_\epsilon R^{\epsilon}\left(\frac{M}{W}\right)^{\frac13}(\sum_T \|f_T\|^2_{L^6(\omega_{B_R})})^\frac12.\end{equation}
	
\end{theorem}

This matches the refined decoupling inequalities appeared in \cite{GIOW19}, and the proof is very similar. We include the proof here for the sake of completeness.

\begin{proof}Like in the proof for the Euclidean case, we shall employ induction on scales via parabolic rescaling. We shall use the microlocal decoupling inequality \eqref{dec} instead of the Bourgain--Demeter decoupling inequality as an input. Since parabolic rescaling is required, we shall prove a stronger theorem which works for all phase function $\phi$ in our class. 
	
We claim that it is enough to consider that case that $\|f\|_{L^p(Q)}$ is approximately a constant for all $R^{1/2}$ cubes $Q\subset Y$. This technique is called dyadic pigeonholing, which is standard in the study of modern harmonic analysis. We will do this reduction a couple of times in this paper, while details are only provided for this one. More precisely, sort $Q\subset Y$ into different groups
$$Y_1,\dots, Y_{N}, Y_{\geq (N+1)},$$
where $N=10d\,\log R$,
$$Y_{j}=\{Q\subset Y: 2^{-j+1}\|f\|_{L^p(Y)}<\|f\|_{L^p(Q)} \leq 2^{-j}\|f\|_{L^p(Y)}\},$$ $$Y_{>N}=\{Q\subset Y: \|f\|_{L^p(Q)}< 2^{-(N+1)}\|f\|_{L^p(Y)}\}.$$
Since by the trivial estimate $\|f\|_{L^p(Y_{>N})}<\frac{1}{10}\|f\|_{L^p(Y)}$, we have
$$\|f\|_{L^p(Y)}\lesssim \sum_{j=1}^{10d\,\log R}\|f\|_{L^p(Y_j)},$$ and therefore there exists $j_0$ such that
\begin{equation}\label{pigeonholing_log}\|f\|_{L^p(Y)}\lesssim (\log R)\,\|f\|_{L^p(Y_{j_0})}.\end{equation}
Recall by our definition $\|f\|_{L^p(Q)}$ is approximately a constant for all $R^{1/2}$-cubes in $Y_{j_0}$. It follows that, if we can prove \eqref{rdec} on $Y_{j_0}$, then \eqref{rdec} holds on $Y$ with another $\log R$ factor, that can be absorbed by the $C_\epsilon R^\epsilon$ term.

From now we assume that $\|f\|_{L^p(Q)}$ is approximately a constant for all $R^{1/2}$ cubes $Q\subset Y$.

	For each $R^{-1/4}$-cap $\sigma$, we consider the collection of fat tubes in $B_R$ with length $R$ and cross-section radius $R^{3/4}$, and central axis being $\gamma^t_\sigma$ for some $t\in \mathbb R$. Let $\mathbb F_\sigma$ be a subcollection of such fat tubes with finite overlaps that covers $B_R$. Let $\mathbb F=\cup_\sigma \mathbb F_\sigma,$ then each fat tube in $\mathbb F$ is associated to one $\sigma=\sigma(\square)$. Then we consider the subcollection of $\mathbb W$
	\[\mathbb W_{\square}:=\{T\in \mathbb W: T\in\mathbb T_\tau \text{ for some $R^{-\frac12}$-cap }\tau\subset \sigma(\square)\text{ and }T\subset \square\}.\]
	
	Now consider $f_{\square}=\sum_{T\in\mathbb W_{\square}}f_T,$ then it is easy to see that $f_{\square}$ has microlocal support in $N_{\phi^\lambda,R^\frac12}^\sigma.$ Notice that every $R^\frac12$-cube $Q$ lies in a unique fat tube $\square$ associated to each cap $\sigma$, and thus $f_\square=f_{\sigma(\square)}$ on each $Q$. By applying Theorem \ref{decoupling-theorem} at scale $R^\frac12$, we get
	\begin{equation}\|f\|_{L^6(Q)}\le C_\epsilon R^{\epsilon}(\sum_{\square} \|f_{\square}\|^2_{L^6(Q)})^\frac12.\end{equation}
	
	Now, we are in a position to begin the induction on scales argument. The base case when $R=1$ is trivial. We assume that \eqref{rdec} is true at scale $R^\frac12$, we shall prove it for the scale $R$. 
	
	Let us decompose $\square$ into $R^\frac12\times R^\frac34$ tubes with central curves $\gamma_\sigma$.  Let $Y_\square$ be the union of the $R^\frac12\times R^\frac34$ tubes which intersect $Y$. Now we do dyadically pigeonhole twice. First, by refining the collection of $Y_\square$, we may assume that all the $R^\frac12\times R^\frac34$ tubes in $Y_\square$ intersect $\sim M'$ tubes $T\in\mathbb W_\square$, for a $\approx 1$ fraction of $Q\subset Y$.  Secondly,  by refining the collection of fat tubes $\square$, we may assume $|\mathbb W_\square|$ is about a constant $W'$ for a $\approx 1$ fraction of $Q\subset Y$.

	Note that the decomposition $f_{\square}=\sum_{T\in\mathbb W_{\square}}f_T,$ on each $Q$ is equivalent to the setup of \eqref{rdec} at scale $R^\frac12$.
	Indeed, after employing the parabolic rescaling introduced in the proof of Lemma \ref{para} to $\square$, we see that each of the $R^\frac12\times R^\frac34$ tubes essentially become a ball of radius $R^\frac14$, and $\square$ essentially becomes a ball of radius $R^\frac12$. Our induction hypothesis then implies that
	\[\|f_\square\|_{L^6(Y_\square)}\le C_\epsilon R^{\epsilon/2}\left(\frac{M'}{W'}\right)^\frac13(\sum_{T\in\mathbb W_\square} \|f_{T}\|^2_{L^6})^\frac12.\]
	Finally, we perform dyadic pigeonhole one last time for the number of fat tubes $\square$ such that $Q\subset Y_\square$. This results in a subset $Y'\subset Y$ which is a union of $R^\frac12$ cubes $Q\subset$ that each lies in $\sim M''$ choices of $\square\in\mathbb B.$ Now for each $Q\subset Y'$, we have
	\[\|f\|_{L^6(Q)}\le C_\epsilon R^{\epsilon} \Big\|\sum_{\square\in \mathbb B:Q\subset Y_\square}{f_\square}\Big\|_{L^6(Q)}.\]
	Invoking decoupling inequality \eqref{dec}, we have
	\[\|f\|_{L^6(Q)}\le C_\epsilon R^{\epsilon} \Big(\sum_{\square\in \mathbb B:Q\subset Y_\square}\|{f_\square}\|_{L^6(Q)}^2\Big)^\frac12.\]
	Noting that the number of terms in the sum is $\sim M''$, if we apply H\"older's inequality, raise everything to the $6$-th power and sum over $Q\subset Y'$, we see that
	\[\|f\|^6_{L^6(Y)}\le C_\epsilon R^{\epsilon} (M'')^2\sum_{\square\in\mathbb B}{\|f_\square\|^6_{L^6(Y_\square)}}.\]
	Using our bound on ${\|f_\square\|^6_{L^6(Y_\square)}}$ from the induction hypothesis and taking account of loss from dyadic pigeonholing, we see that
	\[\|f\|^6_{L^6(Y)}\le C_\epsilon R^{4\epsilon}\left(\frac{M'M''}{W'}\right)^2\sum_{\square\in\mathbb B}(\sum_{T\in\mathbb W_\square} \|f_{T}\|^2_{L^6})^3.\]
	Since $\|f_{T}\|_{L^6}$ is about a constant, $\sum_{T\in\mathbb W_\square} \|f_{T}\|^2_{L^6}=\frac {W' }W\sum_{T\in\mathbb W} \|f_{T}\|^2_{L^6}$. we have
	\[\|f\|^6_{L^6(Y)}\le C_\epsilon R^{4\epsilon}\left(\frac{M'M''}{W'}\right)^2\frac{|\mathbb B|W'}{W}(\sum_{T\in\mathbb W} \|f_{T}\|^2_{L^6})^3.\]
	Noting that $M'M''\le M$ and $|\mathbb B|W'\le W$, we have
	
	\[\|f\|_{L^6(Y)}\le C_\epsilon R^{\frac23\epsilon}\left(\frac{M}{W}\right)^\frac13(\sum_{T\in\mathbb W} \|f_{T}\|^2_{L^6})^\frac12,\]
	which closes the induction thanks to the $R^{-\frac13\epsilon}$ gain.
\end{proof}

\section{Application to Riemannian distance set}
\label{sectiondecouplingtodistance} 

\subsection{$L^2$ method of Mattila--Liu in $\mathbb R^d$. }
In Euclidean spaces, let $d(x, y)=d^x(y)= |x-y|$, and $\dim_{\mathcal H}$ denote the Hausdorff dimension. For any compact set $E\subset \R^d$, and any $0<\alpha< \dim_{\mathcal H} E$, the well known Frostman Lemma (see, e.g. \cite{Mat15}) guarantees that there exists a probability measure $\mu$ on $E$ satisfying
$$\mu(B_r(x))\lesssim r^\alpha,\ \forall\,r>0,\,x\in\R^d.$$

Mattila's approach in \cite{Mat87} is to consider the $L^2$ boundedness of $d_*(\mu\times\mu),$ which would imply  $|\Delta(E)|>0$. This idea was later generalized by Liu \cite{Liu18}, who considered the $L^2$ boundedness of $d^x_*(\mu)$, which would imply $|\Delta_x(E)|>0.$

 Before \cite{GIOW19}, the best known dimensional exponent in $\mathbb R^2$ is $\dim_{\mathcal H} E>4/3$, obtained by plugging Wolff's estimate \cite{W99} i nto the approach of Mattila--Liu. The following result shows $4/3$ cannot be improved via this $L^2$ method. The construction is based on train-track examples, which is described in details in Section $6$ of \cite{GIOW19}. 

\begin{proposition}[Guth, Iosevich, Ou, Wang, 2019] For every $1\leq \alpha < 4/3$ and every positive number $C$, there is a probability measure $\mu$ on $B_1(0)$ with the following properties:
	
	\begin{enumerate}
		\item For any ball $B_r(x)$, $\mu(B_r(x)) \lesssim r^\alpha$.
		
		\item If $d(x,y) := |x-y|$, then
		
		$$ \| d_* (\mu \times \mu) \|_{L^2} > C. $$
		
		\item If $d^x(y) := |x-y|$, then for every $x$ in the support of $\mu$, 
		
		$$ \| d^x_*(\mu) \|_{L^2} > C. $$
		
	\end{enumerate}	
	
\end{proposition}

This shows, to break the $4/3$ barrier, new ideas beyond the $L^2$ method above are required. In \cite{GIOW19}, authors determine good/bad wave packets, decompose $\mu=\mu_{good}+\mu_{bad}$, and then show when $\alpha>5/4$ the contribution from $\mu_{bad}$ is negligible and the $L^2$ method works on $\mu_{good}$. We shall employ the same strategy in the Riemannian case.

\subsection{Riemannian distance set and the pruning procedure.}
Train track examples can be constructed similarly on manifolds, by forming train-tracks along geodesics in a small enough local chart, and thus the direct analog of Proposition 6.1 holds in the setting of two-dimensional Riemannian manifolds. For this reason we cannot expect to obtain the desired result by estimating the $L^2$ norm of the distance measure, so we follow \cite{GIOW19} and eliminate the contribution of the ``rail tracks" using a suitable pruning procedure. This is where we now turn our attention. 

Let $(M,g)$ be a Riemannian surface and $E\subset M$ be a compact set with positive $\alpha$-dimensional Hausdorff measure. Rescaling the metric, if necessary, we can assume that $E$ is contained in the unit geodesic disk of a coordinate patch of $M$, and in this disk, the metric $g$ is very close to the flat metric. Let $E_1$ and $E_2$ be subsets of $E$ with positive $\alpha$-dimensional Hausdorff measure so that the distance from $E_1$ to $E_2$ is comparable to $1$, and $E_j$ is supported  in some geodesic disk $B_j$ of radius $1/100$. Each $E_i$ admits a probability measure $\mu_i$ such that $\supp\, \mu_i\subset E_i$ and $\mu_i(B(x,r))\lesssim r^\alpha.$ Here $B(x,r)$ denotes the geodesic disk of radius $r$.

Let $d_g$ be the Riemannian distance function on $M$. We define the pushforward measure $d_*^x(\mu)=(d_g)_*^x(\mu)$ by
\[\int_{\R} h(t)\,d_*^x(\mu)=\int_M h(d_g(x,y))\,d\mu(y).\]

\vskip.125in 

\subsection{Microlocal decomposition}\label{wave-packet-decomp}
We shall use the geodesic normal coordinates $\{(x_1,x_2)\}$ about a given point $x_0$ in the middle of $E_1$ and $E_2$, such that in this coordinate system, 
$$E_1\subset\{(x_1,x_2): |x_1|\le\frac1{20}, x_2\ge\frac13 \},$$ and 
$$E_2\subset\{(x_1,x_2): |x_1|\le\frac1{20}, x_2\le-\frac13 \}.$$ More precisely, we put $E_1$ and $E_2$ on the 2nd coordinate axis symmetrically. Then we identify the cotangent space at each point on the $x_1$-axis using parallel transport along this axis.

The microlocal decomposition is performed with respect to the geodesic flow transverse to the $x_1$ axis. 

Let $R_0\gg1$, $R_j=2^jR_0$.  Cover the part of the annulus $R_{j-1}\le|\xi|\le R_j$ where $\{(\xi_1,\xi_2):\frac{|\xi_2|}{|\xi|}\ge\frac12\}$ by rectangular blocks $\tau$ with size about $R_j^\frac12\times R_j$, with long direction of each block being the radial direction. Then we need two  big blocks to cover the remaining part of the annulus.  We choose a partition of unity subordinate to this cover, so that
\[1=\psi_0+\sum_{j\ge1}\Big[\sum_{\tau}\psi_{j,\tau}+\psi_{j,-}+\psi_{j,+}\Big],\]
where $\psi_{j,\tau}$ is supported on $\tau$, $\psi_{j,+}$ is a bump function adapted to the set 
\[\{(\xi_1,\xi_2):R_{j-1}\le |\xi|\le R_j,\  0\le\frac{\xi_2}{|\xi|}\le\frac12\},\]
and, similarly, $\psi_{j,-}$ is a bump function adapted to the set 
\[\{(\xi_1,\xi_2):R_{j-1}\le |\xi|\le R_j,\  -\frac12\le\frac{\xi_2}{|\xi|}\le0\}.\]

Let $\delta>0$ be a small constant, and $T_0$ be the geodesic tube of width $1/10$ about the $x_2$-axis.  For each $(j,\tau)$, we look at geodesics $\gamma$ so that $\gamma$ intersects the $x_1$-axis, and the tangent vector of $\gamma$ at the intersection point is pointing in the direction of $\tau$. Now we use geodesic tubes $T_\gamma$ of size $R_j^{-1/2+\delta}\times 1$ about such geodesics to cover $T_0$. Let $\mathbb T_{j,\tau}$ be the collection of all these tubes, and $\mathbb T_j=\cup_{\tau}\mathbb T_{j,\tau}$, $\mathbb T=\cup_{j,\tau}\mathbb T_{j,\tau}$. Let $\eta_T$ be a partition of unity subordinate to this covering, that is, in the unit disk,

\[\sum_{T\in\mathbb T_{j,\tau}}\eta_T=1.\]

Now we need to use a microlocal decomposition for functions supported in $B_1$ which respects the geodesic flow. For function $f$ with $\supp\, f\subset B_1$ and each $(j,\tau)$, let

\[\Psi_{j,\tau}f(x)=\frac{1}{(2\pi)^2}\iint_{\mathbb R^2\times\mathbb R^2} e^{i(x-y)\cdot\xi}{\psi_{j,\tau}(\Phi(x,\xi))}f(y)\,dyd\xi\]
where $\Phi(x,\xi)$ is a smooth function with bounded derivatives. For each $(x,\xi)$, we define $\Phi(x,\xi)$ to be the direction that satisfies {\begin{itemize} \item$|\Phi(x,\xi)|=|\xi|$, \item There exists a unique point $z=(z_1(x,\xi),0)$ on the $x_1$-axis, so that the geodesic connecting $z$ and $x$ has tangent vector ${\Phi(x,\xi)}/|\Phi(x,\xi)|$ at $z$ and $\xi/|\xi|$ at $x$.\end{itemize} Similarly, if we define
\[\Psi_{j,\pm}f(x)=\frac{1}{(2\pi)^2}\iint_{\mathbb R^2\times\mathbb R^2} e^{i(x-y)\cdot\xi}{\psi_{j,\pm}(\Phi(x,\xi))}f(y)dyd\xi\]
and
\[\Psi_{0}f(x)=\frac{1}{(2\pi)^2}\iint_{\mathbb R^2\times\mathbb R^2} e^{i(x-y)\cdot\xi}{\psi_{0}(\Phi(x,\xi))}f(y)dyd\xi,\]
then it is clear that $\Psi_0+\sum_{j\ge1}[\sum_{\tau}\Psi_{j,\tau}+\Psi_{j,-}+\Psi_{j,+}]$ is the identity operator. Denote $A_*(x,\xi)=\psi_*\circ\Phi(x,\xi)$ to be the symbol of the pseudodifferential operator $\Psi_*.$

Now, for each $T\in\mathbb T_{j,\tau}$, define the operator $M_T$ by
\[M_T f:=\eta_T\Psi_{j,\tau}f.\]
Similarly, define $M_0f=\eta_{T_0}\Psi_{0}$, and $M_{j,\pm}f=\eta_{T_0}\Psi_{j,\pm}$. 

\subsection{Good and Bad tubes.}
Denote by $4T$ the concentric tube of four times the radius. We call a tube $T\in\mathbb{T}_{j,\tau}$ bad if 
\[\mu_2(4T)\ge R_j^{-1/2+100\delta},\]
other wise, we say $T$ is good.

Define
\[\mu_{1,good}:=M_0\mu_1+\sum_{T\in\mathbb T, T good} M_T\mu_1.\]

In the Euclidean setting, it is proved in \cite{GIOW19} that the contribution from bad tubes is negligible. Our goal is to generalize this result to the Riemannian setting. We shall prove the following.
\begin{proposition}[Bad Tubes]\label{L1-bad} If $\alpha>1$, and if we choose $R_0$ large enough, then there is a subset $E_2'\subset E_2$ so that $\mu_2(E_2')\ge\frac{999}{1000}$ and for each $x\in E_2'$,
	\[\|d_*^x(\mu_1)-d_*^x(\mu_{1,good})\|_{L^1}<\frac1{1000}.\]
\end{proposition}

On the other hand, we shall prove the following estimate for good tubes.

\begin{proposition}[Good Tubes]\label{L2-Good} If $\alpha>\frac54$, then we have
$$\int ||d^x_*(\mu_{1, good})||_{L^2}^2\,d\mu_2(x)<\infty.$$
\end{proposition}
Now we are ready to prove Theorem \ref{maindistance} using Proposition \ref{L1-bad} and \ref{L2-Good}. The proof of Proposition \ref{L1-bad} and \ref{L2-Good} will be given in Section \ref{L1bad} and \ref{sectionproofmain} respectively. 
\begin{proof}[Proof of Theorem \ref{maindistance}] 

By Proposition \eqref{L2-Good} and  \ref{L1-bad}, there exists $R_0>0$, $x\in E_2$ such that
$$\|d_*^x(\mu_1)-d_*^x(\mu_{1,good})\|_{L^1}<\frac1{1000},$$$$||d^x_*(\mu_{1, good})||_{L^2}^2<\infty.$$

Since $d_*^x(\mu_1)$ is a probability measure on $\Delta_{g,x}(E)$,
\begin{equation}\begin{aligned}\label{enough-to-estimate-good}\int_{\Delta_{g,x}(E)} |d_*^x(\mu_{1, good})|= &\int |d_*^x(\mu_{1, good})|-\int_{\Delta_{g,x}(E)^c} |d_*^x(\mu_{1, good})|\\=&\int |d_*^x(\mu_{1, good})|-\int_{\Delta_{g,x}(E)^c} |d_*^x(\mu_1)-d_*^x(\mu_{1, good})|\\\geq &\int |d_*^x(\mu_1)|-2\int |d_*^x(\mu_1)-d_*^x(\mu_{1, good})|\\\geq &1-\frac{2}{1000}.\end{aligned} \end{equation}

On the other hand, by Cauchy-Schwarz
$$\left(\int_{\Delta_{g,x}(E)} |d_*^x(\mu_{1, good})|\right)^2\leq |\Delta_{g,x}(E)|\cdot ||d^x_*(\mu_{1, good})||_{L^2}^2\lesssim |\Delta_{g,x}(E)|.$$
Hence $|\Delta_{g,x}(E)|>0$ as desired.
\end{proof}

\section{$L^1$ Estimate for Bad Tubes.}
\label{L1bad}
We proceed by considering $d_*^x(M_T\mu_1)$ for each tube $T$. We shall prove the following variable coefficient analogues of Lemma 3.1 and  3.2 in \cite{GIOW19}.
\begin{lemma}\label{RapDec-wave-packet} If $T\in\T$, $\supp f\subset B_1,$ then
	\begin{equation}\|M_Tf\|_{L^1}\lesssim\|f\|_{L^1(2T)}+\ra(R_j)\|f\|_{L^1}\label{e}.
	\end{equation}For $x\in E_2$, and $x\not\in 2T$, we have
	\begin{equation}\|d_*^x(M_T\mu_1)\|_{L^1}\le \ra(R_j).\label{e1}\end{equation}
	Moreover, for any $x\in E_2$,
	\begin{equation}\label{e2}\|d_*^x(M_{j,\pm}\mu_1)\|_{L^1}\le \ra(R_j).\end{equation}
\end{lemma}
\begin{proof} 
	Recall that
	\[M_T \mu_1(y)=\eta_T(y)\frac{1}{(2\pi)^2}\iint e^{i(y-z)\cdot\xi}{A_{j,\tau}(y,\xi)}d\mu_1(z)d\xi \]
	\[=\eta_T(y)\int {\widetilde A_{j,\tau}(y,y-z)}d\mu_1(z).\]
	Here $\widetilde A_{j,\tau}(y,{}\cdot{})$ is the inverse Fourier transform of $A_{j,\tau}$ in the second variable. By uncertainty principle, we note that for each given $y$, $\widetilde A_{j,\tau}(y,{}\cdot{})$ is a smooth Schwartz class function with essential support contained in a rectangle $\mathcal R_y$ centered at 0 of size $R_j^{-1/2}\times R_j^{-1}.$ If $\chi_{2\mathcal R_y}$ denotes the characteristic function of $2\mathcal R_y$, the above implies that $$|\widetilde A_{j,\tau}(y,y-z)|\lesssim R_j^\frac32\chi_{2\mathcal R_y}(y-z)+\ra(R_j)$$ uniformly in $z$. Now the left-hand side of \eqref{e} is
	\begin{align*}&\int\Big|\eta_T(y)\int \widetilde A_{j,\tau}(y,y-z)f(z)dz\Big|dy\\\lesssim &  R_j^\frac32 \int \eta_T(y)|f(z)|\int |\chi_{2R_z}(y-z)|dy\,dz+\ra(R_j)\\
	\lesssim &\int_{2T} |f(z)|dz+\ra(R_j), \end{align*}
	here we have used the fact that $T$ has width $R_j^{-1/2+\delta}\gg R_j^{-1/2}$. The above implies \eqref{e}.
	
	To prove \eqref{e1}, we write
	\[d_*^x(M_T\mu_1)(t)=\int_{S^1(x,t)}M_T\mu_1(y)dl(y),\]
	where $l(y)$ is the formal restriction of $\mu_2$ on the geodesic circle $S^1(x,t)=\{y:d_g(x,y)=t\}$. Thus, $d_*^x(M_T\mu_1)(t)$ is negligible unless $t\approx 1$ due to the separation of $E_1$ and $E_2$.
	Now since 
	\[M_T \mu_1=\eta_T(y)\frac{1}{(2\pi)^2}\iint e^{i(y-z)\cdot\xi}{A_{j,\tau}(z,\xi)}d\mu_1(z)d\xi,\]
	and $|\int d\mu_1|=1$, it suffices to show that, uniformly in $\xi$, we have 
	\[\int_{S^1(x,t)}{\eta_T(y)e^{iy\cdot\xi}}dl(y)\le \ra(R_j).\]
	Now let $y=|y|(\cos\omega,\sin\omega)$ parametrize the geodesic circle $S^1(x,t)$. Then $dl(y)\approx |y|d\omega.$ If we write $\xi=|\xi|(\cos\theta,\sin\theta),$ we can see that
	\[\frac{d}{d\omega}(y\cdot\xi)=|\xi||y|\frac{d}{d\omega}[\cos(\omega-\theta)]=|\xi||y|\sin(\theta-\omega).\]
	Since $x\not\in 2T, y\in T$, we have $|\theta-\omega|\gtrsim R_j^{-1/2+\delta},$ thus $$|\partial_\omega y(\omega)\cdot \xi|\gtrsim R_j^{1/2+\delta}.$$ Since \[|\partial^{(k)}_\omega\eta_T(y(\omega))|\lesssim R_j^{k(1/2-\delta)},\]
	integration by parts yields \eqref{e1}.
	A similar proof using $T_0$ in place of $T$ gives \eqref{e2}.
\end{proof}
A corollary of \eqref{e} is that 
\begin{equation}\label{e3}\|d_*^x(M_T\mu_1)\|_{L^1}\lesssim \mu_1(2T)+\ra(R_j).\end{equation}

\begin{lemma} For $f$ supported in $B_1$, we have
	\[\|f-M_0f-\sum_{T\in\mathbb T}M_Tf-\sum_\pm M_{j,\pm}f\|_{L^1}\lesssim \ra(R_0)\|f\|_{L^1}.\]
\end{lemma}
\begin{proof}
	Note that $\eta_{T_0}\{\Psi_0+\sum_{j\ge1}[\sum_{\tau}\Psi_{j,\tau}+\Psi_{j,-}+\Psi_{j,+}]\}$ is equal to 1 on $T_0$, it suffices to bound 
	\[\|\eta_{T_0}\Psi_{j,\tau}f-\sum_{T\in\T }M_Tf\|_{L^1}\lesssim \ra(R_j)\|f\|_{L^1}.\]
	The left-hand side is 
	\[\|\eta_{T_0}(1-\sum_{T\in\T}\eta_T)\Psi_{j,\tau}f\|_{L^1},\]
	which has rapid decay in $R_j$.
\end{proof}

Now define
\begin{equation}\text{Bad}_j(x):=\bigcup_{T\in\mathbb T_j:\,x\in 2T, T \text{ is bad}}2T.\end{equation}
\begin{lemma}For any $x\in E_2$,
	\[\|d_*^x(\mu_{1,good})-d_*^x(\mu_1)\|_{L^1}\lesssim\sum_{j\ge1}\mu_1(\textnormal{Bad}_j(x))+{\ra}(R_0).\]
\end{lemma}
\begin{proof} If we use Lemma 6.4 and \eqref{e2}, we see that
	\[\|d_*^x(\mu_{1,good})-d_*^x(\mu_1)\|_{L^1}\lesssim\sum_{j\ge1}\sum_{T\in\mathbb T_j, T \text{ bad}}\|d_*^x(M_T\mu_1)\|_{L^1}+\ra(R_0).\]
	By invoking \eqref{e3} and \eqref{e1}, we have
	\[\|d_*^x(\mu_{1,good})-d_*^x(\mu_1)\|_{L^1}\lesssim\sum_{j\ge1}\sum_{T\in\mathbb T_j, x\in 2T, T \text{ bad}}\mu_1(2T)+\ra(R_0),\]
	since there are only finite overlaps between different $2T$'s, the right-hand side is bounded by
	\[\sum_j\mu_1(\text{Bad}_j(x))+\ra(R_0).\]\end{proof}
To estimate the measure of $\text{Bad}_j(x)$, we shall need a generalized version of Orponen's radial projection theorem that is adapted to our microlocal setup.

For a point $y\in E_1$, consider the generalized radial projection map defined on $x\in B_2$
\begin{equation}\pi^y(x)=\frac{\exp_y^{-1}x}{|\exp_y^{-1}x|}\in S^1.\end{equation}
Here $\exp_y^{-1}$ is the inverse of the exponential map centered at $y$, which gives the tangent vector of the geodesic connecting $x$ and $y$. We may assume directions are identified at different base points.

We need the following analog of Theorem 3.7 stated in \cite{GIOW19}. The proof is given in Section \ref{Sec-radial-proj}.
\begin{theorem}
	For every $\alpha>1$ there exists $p(\alpha)>1$ so that 
	\begin{equation}
	\int\|\pi^y_*\mu_2\|_{L^p}^pd\mu_1(y)<\infty.
	\end{equation}
\end{theorem}
As in \cite{GIOW19}, if we denote 
\[\text{Bad}_j:=\{(y,z): \text{there is a bad } T\in\mathbb T_j \text{ so that $2T$ contains $y$ and $z$}\},\]
then we show that the above analog of Orponen's projection theorem implies the following.

\begin{lemma}\label{Badj}
	For each $\alpha>1$, there is a constant $c(\alpha)>0$ so that for each $j\ge1$,
	\[\mu_1\times\mu_2(\text{Bad}_j)\lesssim R_j^{-c(\alpha)\delta}\]
\end{lemma}
\begin{proof}
	Note that 
	\[\mu_1\times\mu_2(\text{Bad}_j)=\int\mu_2(\text{Bad}_j(y))d{\mu_1(y)}.\]
	Suppose that $T\in\mathbb T_j$ is a bad rectangle and $y\in 2T$. Let $Arc(T)$ be the arc of $S^1$ that correspond to the direction of $T$ with length about $R_j^{-1/2+\delta}$. Note that if we take two points $z,\ z'$ on the central geodesic of the tube then $\pi^z(z')\in S^1$ gives the center of the arc. It then follows that $\pi^y(4T)\subset Arc(T)$, and thus
	\[(\pi^y\mu_2)(Arc(T))\ge\mu_2(4T)\ge R_j^{-1/2+100\delta}.\]
	
	Therefore $\pi^y(\text{Bad}_j(y))$ can be covered by arcs $Arc(T)$ of length about $R_j^{-1/2+\delta}$, satisfying the above estimates. By the Vitali covering lemma, we can choose a disjoint subset of the arcs $Arc(T)$ so that  $5Arc(T)$ covers $\pi^y(\text{Bad}_j(y))$, thus we have
	\[|\pi^y(\text{Bad}_j(y))|\lesssim R_j^{-99\delta}.\]
	Now 	\begin{align*}\mu_1\times\mu_2(\text{Bad}_j)&=\int\mu_2(\text{Bad}_j(y))d{\mu_1(y)}\le\int\left(\int_{\pi^y(\text {Bad}_j(y))}\pi^y\mu_2\right)d\mu_1(y)\\
	&\le |\pi^y(\text{Bad}_j(y))|^{1-1/p}\int\|\pi^y\mu_2\|_{L^p}d\mu_1\lesssim R_j^{-c(\alpha)\delta}.
	\end{align*}
\end{proof}
Now we can use Lemma \ref{Badj} to prove Proposition \ref{L1-bad}, the proof is identical to that of \cite{GIOW19}.
\begin{proof}[Proof of Proposition \ref{L1-bad}]
	Recall 	\[\mu_1\times\mu_2(\text{Bad}_j)=\int\mu_2(\text{Bad}_j(x))\,d{\mu_1(x)}.\]
	For each $j\ge1$, we can choose $S_j\subset E_2$ so that $\mu_2(S_j)\le R_j^{-(1/2)c(\alpha)\delta}$, and for all $x\in E_2\setminus S_j$,
	\[\mu_1(\text{Bad}_j(x))\lesssim R_j^{-(1/2)c(\alpha)\delta}.\] 
	Let $E_2'=E_2\setminus\cup_{j\ge1}S_j$, thus $\mu_2(E_2')\ge 1-1/1000$ if $R_0$ is sufficiently large. Then for $x\in E_2'$, we have
	\[\|d_*^x(\mu_1)-d_*^x(\mu_{1,good})\|_{L^1}\lesssim \sum_{j\ge1}\mu_1(\text{Bad}_j(x))+\ra(R_0)\lesssim R_0^{-(1/2)c(\alpha)\delta},\]
	which can be arbitrarily small if we choose $R_0$ to be large enough.
\end{proof}

\section{$L^2$ Estimate for Good Tubes}
\label{sectionproofmain} 
In this section, we prove Proposition \ref{L2-Good}.
We consider the $L^2$ quantity
$$\int ||d^x_*(\mu_{1, good})||_{L^2}^2\,d\mu_2(x).$$
Since $\mu_{1, good}\in C_0^\infty(B_1(0))$, one can see that
$$d^x_*(\mu_{1, good})(t)\approx \int_{ d_g(x,y)=t} \mu_{1, good}(y)\,dy = \iint e^{-2\pi i ( d_g(x,y)-t)\tau}\,d\tau \mu_{1, good}(y)\,dy.$$

By Plancherel,
$$||d^x_*(\mu_{1, good})||_{L^2}^2= \int\left|\int e^{-2\pi i \lambda  d_g(x,y)}\, \mu_{1, good}(y)\,dy\right|^2\,d\lambda.$$

Denote $R_j=2^j R_0$. Then it suffices to show that for each $j$,
$$\Big|\int\int_{\lambda\approx R_j}\left|\int e^{-2\pi i \lambda  d_g(x,y)}\, \mu_{1, good}(y)\,dy\right|^2\,d\lambda\,d\mu_2(x)\Big|\lesssim 2^{-j\varepsilon}$$
for some $\varepsilon>0$.

Recall that
\[\mu_{1,good}:=M_0\mu_1+\sum_{T\in\mathbb T, T\,good} M_T\mu_1.\]

Since $\lambda\approx R_j$, by standard integration by parts argument it is equivalent to consider
\begin{equation}\label{lambda=R}\int\int_{\lambda\approx R_j}\left|\int e^{-2\pi i \lambda  d_g(x,y)}\, \mu^j_{1, good}(y)\,dy\right|^2\,d\lambda\,d\mu_2(x),\end{equation}
where
\[\mu^j_{1,good}:=\sum_{T\in\mathbb T_j, T\,good} M_T\mu_1.\]

Denote
$$F^\lambda(x)=\int e^{-2\pi i \lambda  d_g(x,y)}\, \mu^j_{1, good}(y)\,dy,$$
$$F^\lambda_T(x)=\int e^{-2\pi i \lambda  d_g(x,y)}\, M_T\mu_1(y)\,dy.$$
Then $F^\lambda=\sum_{T\in\mathbb T_j, T\,good} F^\lambda_T$. Also denote $R=\lambda^{1-2\delta}\approx R_j^{1-2\delta}$. Then each $T$ is a $R^{-1/2}\times 1$ geodesic tube.

We claim that this decomposition coincides with the wave-packet decomposition in Section \ref{wave-packet-decomp}.  If we recall the definition of $M_T$ and integrate by parts, we see that $F^\lambda_T$ can be replaced by $\chi_{2T}\cdot F^\lambda_T$ at the expense of a $\ra(R_j)$ term, so we may assume that $F^\lambda$ is supported in the unit ball $B_1(0)$. To check the microlocal support, notice
$$F^\lambda_T(x)=\int\int_{2T} e^{-2\pi i (z-x)\cdot\xi } \int_{2T} e^{-2\pi i \lambda d_g(z,y)}\, M_T\mu_1(y)\,dy\,dz\,d\xi.$$
By integration by parts in $z$, one can see that $\xi$ essentially lies in the $R^{2\delta}$-neighborhood of
$$\{\lambda\nabla_z d_g(z,y): y\in 2 T\}.$$
Since $z$ lies in $2T$ as well, both $y,z$ lie in a $R^{-1/2}$-neighborhood of the central geodesic of $T$, thus for each fixed $z$, $\nabla_z d_g(z,y)$ must lie in a $R^{-1/2}$-cap in $S^1$ determined by $T$, as desired.

By dyadic pigeonholing, it is enough to consider tubes $T\in \mathbb{W}_\beta$ where $||F_T||_{L^6}\sim \beta$. Denote $W=\#(\mathbb{W}_\beta)$.

We fix $\lambda$ in \eqref{lambda=R} and integrate over $\mu_2(x)$ first. Let $\psi_{R_j}\in C_0^\infty$ which is nonnegative, equal to $1$ on $B_{10R_j}(0)$ and $0$ outside $B_{20R_j}(0)$. Since we are assuming that $F^\lambda$ is supported on the unit ball, and $\widehat{F^\lambda}(\xi)=\ra(|\xi|)$ outside $B_{10R_j}(0)$. It follows that
$$F^\lambda = \left(\widehat{F^\lambda}\cdot\psi_{R_j}\right)^{\vee}+\ra(R_j)=F^\lambda*\widehat{\psi_{R_j}}+\ra(R_j),$$
and therefore
\begin{equation}\begin{aligned}\int |F^\lambda(x)|^2\,d\mu_2(x)&= \int \left|F^\lambda*\widehat{\psi_{R_j}}(x)\right|^2 \,d\mu_2(x)+\ra(R_j)\\&\lesssim \int |F^\lambda|^2*|\widehat{\psi_{R_j}}|(x) \,d\mu_2(x)+\ra(R_j)\\&=\int_{B_1(0)} |F^\lambda(x)|^2\,|\widehat{\psi_{R_j}}|*\mu_2(x)\,dx+\ra(R_j).\end{aligned}\end{equation}
Denote $\mu_{2, R_j}=|\widehat{\psi_{R_j}}|*\mu_2$. Notice although the support of $\mu_{2, R_j}$ is not compact, we can still work on $B_1(0)$ as $F^\lambda$ has compact support.

Decompose $B_1(0)$ into $R^{-1/2}$-squares $Q$. By dyadic pigeonholing again it is enough to consider
$$\mathcal{Q}_{\gamma, M}=\{Q: \mu_{2, R_j}(Q)\sim \gamma, Q\text{ intersects }\sim M\text{ tubes }T\in\mathbb{W}\}.$$
Denote
$$Y_{\gamma, M} = \bigcup_{Q\in \mathcal{Q}_{\gamma, M}} Q.$$

By H\"older's inequality,
\begin{equation}\label{Holder-to-apply-decoupling}\int_{Y_{\gamma, M}} |F^\lambda(x)|^2\,\mu_{2, R_j}(x)\,dx\leq \left(\int_{Y_{\gamma, M}} |F^\lambda(x)|^6 \,dx\right)^{1/3}\left(\int_{Y_{\gamma, M}}\mu_{2, R_j}(x)^{3/2}\,dx\right)^{2/3}.\end{equation}

The first factor can be estimated by the refined microlocal decoupling inequality \eqref{refined-decoupling-theorem}:
\begin{equation}\label{apply-refined-decoupling}||F^\lambda||_{L^6(Y_{\gamma, M})}\leq C_\epsilon R^{\epsilon} \left(\frac{M}{W}\right)^{\frac13}\Big(\sum_{T\in\mathbb{W}_\beta} \|F^\lambda_T\|^2_{L^6}\Big)^\frac12.\end{equation}

For the second factor, by the ball condition on $\mu_2$ and the rapid decay of $|\widehat{\psi_{R_j}}|$ outside $B_{R_j^{-1}}$, we have
\begin{equation}\label{stupid-upper-bound}\mu_{2, R_j}(x)=\int |\widehat{\psi_R}(x-y)|\,d\mu_2(y)\lesssim R_j^2\cdot\mu_2(B_{R_j})+\ra(R_j)\lesssim R_j^{2-\alpha}.\end{equation}
Therefore
\begin{equation}\label{infty_norm_mu_2}\int_{Y_{\gamma, M}}\mu_{2, R_j}(x)^{3/2}\,dx\lesssim R_j^{\frac{2-\alpha}{2}}\cdot \mu_{2, R_j}(Y_{\gamma, M}).\end{equation}
\begin{lemma}\label{doubling-argument}For any $\gamma$, $M$,
	$$\mu_{2, R_j}(Y_{\gamma, M})\lesssim \frac{W \cdot R_j^{-1/2+100\delta}}{M}.$$
\end{lemma}
The proof is the same as that of \cite[Lemma 5.4]{GIOW19}. We give the proof for the sake of completeness.
\begin{proof}[Proof of Lemma \ref{doubling-argument}]
	This is a double counting argument. Consider
	$$I=\{(Q, T)\in\mathcal{Q}_{\gamma, M}\times \mathbb{W}_\beta : Q\text{ intersects } T\}.$$
	Since each tube $T\in\mathbb{W}_\beta$ is good, we have $$\mu_{2, R_j}(2T)\lesssim \mu_2(4T)\lesssim R_j^{-1/2+100\delta}.$$ Therefore each $T$ intersect $\lesssim \gamma^{-1}\cdot R_j^{-1/2+100\delta}$ many cubes $Q\in\mathcal{Q}_{\gamma, M}$. It implies
	$$\#(I)\lesssim \gamma^{-1}\cdot R_j^{-1/2+100\delta}\cdot W.$$

	On the other hand, each $Q\in \mathcal{Q}_{\gamma, M}$ intersects $M$ tubes $T$. Therefore 
	$$\#(I)\geq\#(\mathcal{Q}_{\gamma, M})\cdot M. $$
	
	Comparing these bounds for $I$, one gets 
	$$\#(\mathcal{Q}_{\gamma, M})\cdot\gamma\lesssim \frac{W \cdot R_j^{-1/2+100\delta}}{M}.$$
	
	Hence the lemma follows since $\mu_2(Q)\sim \gamma$ for each $Q\in\mathcal{Q}_{\gamma, M}$.

\end{proof}

Put \eqref{Holder-to-apply-decoupling}, \eqref{apply-refined-decoupling}, \eqref{infty_norm_mu_2} and Lemma \ref{doubling-argument} together, it follows that
\begin{equation}\label{almost-there}\begin{aligned}\int_{\lambda\approx R_j}\int |F^\lambda(x)|^2\,d\mu_2(x)\,d\lambda\lesssim_\delta &R_j^{O(\delta)+\frac{1-\alpha}{3}}\sum_{T\in \mathbb{T}_j} \int_{\lambda\approx R_j}\|F^\lambda_T\|^2_{L^6}\,d\lambda\\\lesssim_\delta & R_j^{O(\delta)+\frac{1-\alpha}{3}}\cdot |T|^{1/3} \sum_{T\in \mathbb{T}_j}\int_{\lambda\approx R_j}||F^\lambda_T||^2_{L^\infty}\,d\lambda,\end{aligned}\end{equation}
where the last inequality follows since $F^\lambda_T$ is essentially supported on $2T$.

Notice
$$F^\lambda_T(x)= \int \left(\int_{2T} e^{-2\pi i (\lambda  d_g(x,y)-y\cdot\xi)}\,dy\right) \widehat{M_T\mu_1}(\xi)\,d\xi.$$
This implies that $\xi$ essentially lies in the $\lambda/R\approx R_j^{2\delta}$-neighborhood of $\lambda\tau$, where $\tau\subset S_y$ is a $R^{-1/2}$-cap, thus
$$\int_{\lambda\approx R_j}||F^\lambda_T||^2_{L^\infty}\,d\lambda\lesssim \int_{\lambda\approx R_j} \left(\int_{2T}\int |\widehat{M_T\mu_1}(\xi)|\,\psi^\tau_{d^\lambda, R}(\lambda y,\xi/\lambda)\,d\xi\,dy \right)^2d\lambda,$$
where $\psi_{d^\lambda, R}^\tau$ is defined in \eqref{partition}. Then by Cauchy-Schwarz it is bounded from above by
$$\int_{\lambda\approx R_j} \left(\int_{2T}\int |\widehat{M_T\mu_1}(\xi)|^2\,\psi^\tau_{d^\lambda, R}(\lambda y,\xi/\lambda)\,d\xi\,dy \int_{2T}\int \psi^\tau_{d^\lambda, R}(\lambda y,\xi/\lambda)\,d\xi\,dy \right)\,d\lambda.$$

For each fixed $y\in 2T$,
$$\int \psi^\tau_{d^\lambda, R}(\lambda y,\xi/\lambda)\,d\xi\lesssim R_j^{2\delta}\cdot \lambda\cdot R^{-1/2},$$
thus
$$\int_{2T}\int \psi^\tau_{d^\lambda, R}(\lambda y,\xi/\lambda)\,d\xi\,dy\lesssim R_j^{2\delta}\cdot \lambda\cdot R^{-1/2}\cdot |T|\lesssim R_j^{2\delta}\cdot \lambda/R\lesssim R_j^{O(\delta)}.$$

Also notice that $|\widehat{M_T\mu_1}(\xi)|$ is independent in $\lambda$, therefore 
$$\int_{\lambda\approx R_j}||F^\lambda_T||^2_{L^\infty}\,d\lambda\lesssim R_j^{O(\delta)}  \int |\widehat{M_T\mu_1}(\xi)|^2\,\left(\int_T\int_{\lambda\approx R_j}\psi^\tau_{d, \lambda}(\lambda y,\xi/\lambda)\,d\lambda\,dy\right)\,d\xi.$$

By definition of $\psi^\tau_{d^\lambda, R}$ in \eqref{partition} one can see that uniformly in $y\in 2T$, 
$$\int_{\lambda\approx R_j}\psi^\tau_{d^\lambda, R}(\lambda y,\xi/\lambda)\,d\lambda\lesssim R_j^{2\delta}. $$
It follows that
$$\int_{\lambda\approx R_j}||F^\lambda_T||^2_{L^\infty}\,d\lambda\lesssim R_j^{O(\delta)}\cdot |T|\cdot \int |\widehat{M_T\mu_1}(\xi)|^2\,d\xi.$$
Recall that $M_T=\eta_T\Psi_{j,\tau}$. Noticing that $\Psi_{j,\tau}$ is bounded on $L^2$ as it behaves like a $0$-th order pseudodifferential operator, we see that $M_T$ is also bounded on $L^2$. Now by applying Plancherel's theorem twice, we have 
$$\sum_{T\in\mathbb{T}_j} \int_{\lambda\approx R_j}||F^\lambda_T||^2_{L^\infty}\,d\lambda\lesssim R_j^{O(\delta)}\cdot R^{-1/2} \int_{|\xi|\approx R_j} |\hat\mu_1(\xi)|^2\,d\xi\lesssim R_j^{O(\delta)+3/2-\alpha} I_{\alpha-\delta}(\mu_1), $$
where 
$$I_{s}(\mu_1)=\iint |x-y|^{-s}\,d\mu_1(x)\,d\mu_1(y)=c_{s}\int|\hat\mu_1(\xi)|^2\,|\xi|^{-2+s}\,d\xi$$ 
denotes the energy integral which is well known to be finite for any $s\in(0, \alpha)$ (see e.g. \cite[Section 2.5, 3.5]{Mat15}).

Plug this estimate into \eqref{almost-there}, with $|T|\approx R^{-1/2}$:
$$\begin{aligned}\int_{\lambda\approx R_j}\int |F^\lambda(x)|^2\,d\mu_2(x)\,d\lambda\leq &C_\delta\, R_j^{O(\delta)+\frac{1-\alpha}{3}-1/6+3/2-\alpha}=C_\delta \,R_j^{O(\delta)+\frac{5-4\alpha}{3}}\end{aligned},$$
where the power is negative if $\alpha>5/4$ and $\delta>0$ is small enough. The proof is complete.

  \section{Radial projections on manifolds}\label{Sec-radial-proj}
  In Euclidean spaces, denote the radial projection centered at $y\in\R^d$ by
  $$\pi^y(x)=\frac{x-y}{|x-y|}: \R^d\backslash\{y\}\rightarrow S^{d-1}.$$
  
  The following estimate due to Orponen plays an important role in recent work on Falconer distance conjecture \cite{KS18}, \cite{GIOW19}, \cite{Shm18}. 
 \begin{theorem}[{\cite[(3.6)]{Orp19}}]\label{Orponen-radial-proj}
  	Given compactly supported Borel measures $\mu$, $\nu$ in $\R^d$ such that $\supp(\mu)\cap\supp(\nu)=\emptyset$, $I_{s}(\mu), I_{t}(\nu)<\infty$, with $<2(d-1)-s<t<d-1$. Then for any 
  	$$1<p\leq \min\{2-\frac{t}{d-1}, \frac{t}{2(d-1)-s}\}$$
  	we have
  	\begin{equation}\int ||\pi^y_*(\mu)||^p_{L^p(S^{d-1})}\,d\nu(y)\lesssim I_{t}(\nu)^{1/2}\cdot I_{s}(\mu)^{p/2} <\infty.\end{equation}	
  \end{theorem}
  
  In this paper we prove an analog of this result on $2$-dimensional Riemannian manifolds, where the radial projection is defined by
  $$\pi^y(x) = \frac{\exp^{-1}_y x}{|\exp^{-1}_y x|}\in S^{1}.$$

  \begin{theorem}\label{radial-proj-thm}
  	Given compactly supported Borel measures $\mu$, $\nu$ on a $2$-dimensional Riemannian manifold such that $\supp(\mu)\cap\supp(\nu)=\emptyset$, $I_{s}(\mu), I_{t}(\nu)<\infty$, with $2-s<t<1$. Then for any 
  	$$1<p\leq \min\{2-t, \frac{t}{2-s}\}$$
  	we have
  	\begin{equation}\int ||\pi^y_*(\mu)||^p_{L^p(S^{1})}\,d\nu(y)\lesssim I_{t}(\nu)^{1/2}\cdot I_{s}(\mu)^{p/2} <\infty.\end{equation}	
  \end{theorem}
    We believe similar results still hold in higher dimensions. However, in the absence of a good coordinate system in higher dimensions, extra efforts are needed to extend the proof, which would certainly make the argument long and tedious. As we only use radial projections on $2$-dimensional manifolds, and this radial projection theorem is not the main contribution of this paper, we choose to only state and prove the $2$-dimensional version. In fact, we suspect that in higher dimensions we may need to introduce extra concepts and terminologies from differential geometry that are unfamiliar to many readers of this paper. We plan to address this problem in a sequel.

  The proof of Theorem \ref{radial-proj-thm} is almost the same as Orponen's. The only difference is, in Euclidean spaces it is reduced to classical estimates on orthogonal projections, while on manifolds we reduce it to estimates of generalized projections due to Peres and Schlag \cite{PS00}. 
   
  \subsection{Orthogonal projections on manifolds}
  In \cite{PS00}, a very broad class of maps, called generalized projections, is studied. We shall show that local orthogonal projections on manifolds lie in this class.

  We choose a geodesic $\gamma_0$ and work in the Fermi  normal coordinates about $\gamma_0$, so that $\gamma_0=\{(x_1,0)\}$ in this coordinate system.  Then for any $u\in\R$ and $(\cos\theta, \sin\theta)\in S^{1}$, denote by $\pi_\theta(x)=u$ the ``orthogonal projection on manifolds" if there exists $t\in\R$ such that $\exp_{(u,0)}t(\cos\theta, \sin\theta)=x$. This projection map is always well-defined as long as we work within a small enough local chart. We will also assume that $x_2$ is bounded away from $0$. Note that in this case we must have $$(\cos\theta,\sin\theta)=-\dfrac{\nabla_y{d_g(x,(\pi_\theta(x),0))}}{|\nabla_y{d_g(x,(\pi_\theta(x),0))}|}.$$
  Since we are working in the Fermi coordinates about $\gamma_0=\{(x_1,0)\}$, we have that $g_{ij}(u,0)=\delta_{ij}$, where $(g_{ij})$ denotes the metric matrix, and thus $\nabla_y{d_g(x,(u,0))}$ must be a unit vector. Therefore
\begin{equation}\label{diff_unit_vector}(\vec{\omega}\cdot\nabla_x)\nabla_yd_g(x, (u,0))\perp \nabla_y d_g(x, (u,0)),\ \forall\,\vec\omega\in S^1\end{equation}
and by our definition of $\pi_\theta$ we have
 \begin{equation}\label{unit_vector}(\cos\theta, \sin\theta)=-\nabla_y{d_g(x,(\pi_\theta(x),0))}.\end{equation}
The formula \eqref{unit_vector} significantly simplifies our computations. Unfortunately, a higher dimensional analog is not generally available.

Since $\vec\omega$ is arbitrary in \eqref{diff_unit_vector}, we see that $(\cos\theta, \sin\theta)$ must be contained in the kernel of the rank 1 matrix $A_\theta:=\nabla^2_{xy}d_g(x, (\pi_\theta(x),0))$, that is
  \begin{equation}\label{rank1}A_\theta\cdot(\cos\theta,\sin\theta)^T=(0,0)^T.\end{equation}
  If we fix $x$, take $\partial_\theta$ on both sides of \eqref{rank1}, we have
  \begin{equation}\label{rank1 diff}\partial_\theta A_\theta\cdot(\cos\theta,\sin\theta)^T=-A_\theta\cdot(-\sin\theta,\cos\theta)^T,\end{equation}
  Since $A_\theta$ has rank 1 and $(\cos\theta,\sin\theta)^T$ is in its kernel, both sides of \eqref{rank1 diff} have norm $\gtrsim 1$.
  
 To apply estimates in \cite{PS00}, we need to check the regularity condition and the transversality condition. From the local smoothness of the exponential map one concludes that all derivatives of $\pi_\theta(x)$ are bounded, so the regularity condition is satisfied. It remains to check the transversality condition on $\pi_\theta$, namely  \begin{equation}\label{transverality_condition}|\vec{\omega}\cdot\nabla_x\pi_\theta|\ll 1\implies|\partial_\theta (\vec{\omega}\cdot\nabla_x\pi_\theta)|\gtrsim 1,\ \forall\,\vec{\omega}\in S^{1}.\end{equation}
Notice these conditions are obvious under the Euclidean metric where $\pi_\theta(x)=x_1+x_2\cot \theta$. 

By differentiating both sides of \eqref{unit_vector} in $\theta$, one can see $$|\partial_{y_1}\nabla_yd_g|,|\partial_\theta\pi_\theta|\gtrsim 1.$$ Also we can take $\nabla_x$ on both sides of \eqref{unit_vector} to have
  \begin{equation}\label{prepare_for_partial_theta}A_\theta + \partial_{y_1}\nabla_yd_g(x,(\pi_\theta(x), 0))^T\cdot\nabla_x\pi_\theta(x)=0.\end{equation}
As $A_\theta$ is not a $0$ matrix, one concludes $|\nabla_x\pi_\theta|\gtrsim 1$. If we multiply both sides of \eqref{prepare_for_partial_theta} by $(\cos\theta,\sin\theta)^T$ from the right, then by \eqref{rank1} and $|\partial_{y_1}\nabla_yd_g|\gtrsim 1$, it follows that $\nabla_x\pi_\theta\cdot(\cos\theta,\sin\theta)=0$. Therefore 
 \begin{equation}\label{theta_omega}|\vec{\omega}\cdot\nabla_x\pi_\theta|\ll 1\implies |\vec\omega-(\cos\theta,\sin\theta)|\ll1.\end{equation}Now we fix $x$ and take $\partial_\theta$ on both sides of \eqref{prepare_for_partial_theta} to get
 \begin{equation}\label{partial_theta}\partial_\theta A_\theta + (\partial_\theta\pi_\theta)(\partial^2_{y_1}\nabla_yd_g)^T\cdot(\nabla_x\pi_\theta)+(\partial_{y_1}\nabla_yd_g)^T\cdot(\partial_\theta\nabla_x\pi_\theta)=0.\end{equation}
 If we multiply it by $\vec\omega^T$ from the right, then we have
  \begin{equation}\label{partial_theta_dot}\partial_\theta A_\theta\cdot\vec{\omega}^T  + (\partial_\theta\pi_\theta)(\vec{\omega}\cdot\nabla_x\pi_\theta)(\partial^2_{y_1}\nabla_yd_g)^T+(\partial_\theta(\vec{\omega}\cdot\nabla_x\pi_\theta))(\partial_{y_1}\nabla_yd_g)^T=0.\end{equation}
By \eqref{theta_omega}, if $|\vec{\omega}\cdot\nabla_x\pi_\theta|\ll 1$ then $\vec{\omega}$ is close to $(\cos\theta, \sin\theta)$. Therefore the first term has norm about 1 by \eqref{rank1 diff}. Hence
 \eqref{transverality_condition} holds and the transversality condition is checked.

  With the regularity condition and the transversality condition, it is known that (see \cite{PS00}, or \cite[Chapter 18]{Mat15})
  \begin{itemize}
  	\item if $I_\alpha(\mu)<\infty$, then
  	\begin{equation}
  	\label{l2-generalized-proj} 
  	\int_{S^{1}} |\widehat{(\pi_\theta)_*\mu}(\xi)|^2\,d\theta\leq C(\alpha) \,I_\alpha(\mu)\,(1+|\xi|)^{-\alpha};
  	\end{equation}
  	\item for any measure $\sigma$ on $S^{1}$ satisfying $$\sigma(B(\theta, r))\lesssim r^s,\ \forall\,\theta\in S^{1}, r>0,$$ we have
  	$$\int I_{\beta}((\pi_\theta)_*\mu)\,d\sigma(\theta)\lesssim I_\beta(\mu),\  \forall\,0<\beta<s.$$
  \end{itemize}

  In particular, since by H\"older
  $$\int_{B(\theta_0,r)} f(\theta)\,d\mathcal{H}^{1}(\theta)\lesssim r^{1/p}||f||_{p'}, $$
  it follows that for any $\tau\in(0,1)$,
  \begin{equation}\label{Lp-energy}\left(\int\left|I_{\tau}((\pi_\theta)_*\mu)\right|^p d\mathcal{H}^{1}(\theta)\right)^{1/p}\lesssim I_\tau(\mu),\ \forall\, p\in [1, 1/\tau).\end{equation}
  
  \subsection{Proof of Theorem \ref{radial-proj-thm}}
  We follow Orponen's argument in \cite[Section 3]{Orp19}. We may assume $\mu, \nu\in C_0^\infty$, then the general case follows by a standard limit argument (see \cite{Orp19} for details). 
  
  The first step is to reduce radial projections to orthogonal projections: for any $p>0$,
  \begin{equation}\label{reduce-to-orthogonal-proj}\int ||\pi^y_*(\mu)||^p_{L^p(S^{1})}\,d\nu(y) \approx \iint |(\pi_\theta)_*\mu(u)|^p (\pi_\theta)_*\nu(u)\,du\,\mathcal{H}^{1}(\theta).\end{equation}
  where the implicit constant only depends on $dist(\supp(\mu), \supp(\nu))$. 
  
  To see this, for any $f\in C(S^{1})$,
  $$\int f(e)\,d\pi^y_*\mu(e) = \int f\Big(\frac{\exp_y^{-1}x}{|\exp_y^{-1}x|}\Big)\,\mu(y)\,dy.$$
  
  Since $\text{dist}(x,y)\approx 1$, by polar coordinates in $T_y M$ it approximately equals
  $$\int_{S^{1}} f(e) \int \mu(\exp_y t e)\,dt\,d\mathcal{H}^{1}(e).$$
  Therefore as a function
  $$\pi^y_*\mu(e)\approx \int \mu(\exp_y(te))\,dt=\int \mu(\exp_{(\pi_{\theta(y,e)},0)}t\theta(y,e))\,dt,$$
  where $\theta({}\cdot{}, {}\cdot{})$ is as defined in \eqref{def-theta}.
  
  Since the exponential map is a local diffeomorphism on the tangent bundle, the map that sends $y,e$ to $u=\pi_{\theta(y,e)}, \theta=\theta(y,e)$ is differentiable, with Jacobian $\approx 1$. Therefore
  $$\begin{aligned}&\int ||\pi^y_*(\mu)||^p_{L^p(S^{1})}\,d\nu(y)\\\approx & \iint\left|\int \mu(\exp_{(u,0)}t\theta)\,dt\right|^p \left(\int \nu(\exp_{(u,0)} t\theta)\,dt\right)du\, d\mathcal{H}^{1}(\theta)\\\approx&\iint\left|(\pi_\theta)_*\mu (u) \right|^p \,(\pi_\theta)_*\mu (u)\, du\,d\mathcal{H}^{1}(\theta).\end{aligned}$$
  
  Now it suffices to consider the right-hand side of \eqref{reduce-to-orthogonal-proj}. It is known that (see \cite[Lemma 3.4]{Orp19}) if $2-s+\epsilon\in(0, 1)$, then
  $$||g||_{L^1(\lambda)}\lesssim \sqrt{I_{2-s+\epsilon}(\lambda)}\,||g||_{H^{(s-\epsilon-1)/2}}$$
  for any compactly supported measure $\lambda$ on $\R$ and any continuous function $g\in H^\sigma$, where
  \begin{equation}\label{lemma-3.4-in-Orp19}||g||_{H^{(s-\epsilon-1)/2}}:=\left(\int |\hat{g}(\xi)|^2\,|\xi|^{s-\epsilon-1}\,d\xi\right)^{1/2}.\end{equation}
  
  Fix $f\in L^{p'}((\pi_\theta)_*\mu)$, $||f||_{L^{p'}((\pi_\theta)_*\mu)}=1$. Since there exists $\epsilon>0$ such that $2-s+\epsilon\in(0,1)$ and $1<p\leq t/(2-s+\epsilon)$, by H\"older's inequality twice we have
  $$I_{2-s+\epsilon}(f\,d(\pi_\theta)_*\nu)=\iint \frac{f(x)\,f(y)\,d(\pi_\theta)_*\nu(x)\,d(\pi_\theta)_*\nu(y)}{|x-y|^{2-s+\epsilon}}\lesssim I_{t}((\pi_\theta)_*\nu)^{1/p}.$$
  
  Then with $g=(\pi_\theta)_*\mu$, $\lambda=f\,d(\pi_\theta)_*\nu$ in \eqref{lemma-3.4-in-Orp19},
  $$\begin{aligned}\int (\pi_\theta)_*\mu\cdot f\,d(\pi_\theta)_*\nu\lesssim & \sqrt{I_{2-s+\epsilon}(f\,d(\pi_\theta)_*\nu)}\ ||(\pi_\theta)_*\mu||_{H^{(s-\epsilon-1)/2}} \\ \lesssim & I_{t}((\pi_\theta)_*\nu)^{1/2p}\cdot \left(\int |\widehat{(\pi_\theta)_*\mu)}(\xi)|^2\,|\xi|^{s-\epsilon-1}\,d\xi\right)^{1/2}.\end{aligned}$$
  
  It follows that
  $$\left(\int\left|(\pi_\theta)_*\mu (u) \right|^p \,(\pi_\theta)_*\mu (u)\, du\right)^{1/p}\lesssim I_{t}((\pi_\theta)_*\nu)^{1/2p}\cdot \left(\int |\widehat{(\pi_\theta)_*\mu)}(\xi)|^2\,|\xi|^{s-\epsilon-1}\,d\xi\right)^{1/2}. $$
  
  Now it remains to show, for any $h\in L^{p'}(S^{1})$, $||h||_{L^{p'}(S^{1})}=1$, 
  $$\begin{aligned}&\int_{S^{1}}I_{t}((\pi_\theta)_*\nu)^{1/2p}\cdot \left(\int |\widehat{(\pi_\theta)_*\mu)}(\xi)|^2\,|\xi|^{s-\epsilon-1}\,d\xi\right)^{1/2}\cdot h(\theta)\,d\mathcal{H}^{1}(\theta)\\\lesssim & I_{t}(\nu)^{1/2p}\cdot I_{s}(\mu)^{1/2}.\end{aligned} $$
  By Cauchy-Schwarz the left-hand side is bounded from above by
  $$\left(\int_{S^{1}}I_{t}((\pi_\theta)_*\nu)^{1/p}\cdot h(\theta)^2\,d\mathcal{H}^{1}(\theta)\right)^{1/2}\left(\iint |\widehat{(\pi_\theta)_*\mu)}(\xi)|^2\,|\xi|^{s-\epsilon-1}\,d\xi\,d\mathcal{H}^{1}(\theta)\right)^{1/2}.$$
  Then the second factor is $\lesssim I_s(\mu)^{1/2}$ by \eqref{l2-generalized-proj}, and the first factor is 
  $$\lesssim \left(\int_{S^{1}}I_{t}((\pi_\theta)_*\nu)\cdot h(\theta)^p\,d\mathcal{H}^{1}(\theta)\right)^{1/2p}\cdot ||h||^{1/2}_{L^{p'}}\lesssim I_t(\nu)^{1/2p}$$
  by H\"older's inequality and \eqref{Lp-energy}. The proof is complete.

\end{document}